\documentclass[a4paper,12pt]{scrartcl}
\usepackage{amsmath,amssymb,latexsym,amsthm,enumerate,mathrsfs}
\usepackage{xcolor}
\usepackage[nocompress]{cite}
\usepackage{dsfont}
\usepackage{enumitem}

\usepackage{mathtools}

\usepackage[T1]{fontenc}
\usepackage[utf8]{inputenc}
\usepackage[english]{babel}

\usepackage{hyperref}
\usepackage[capitalise,noabbrev]{cleveref}

\newcommand*{\wh}{\widehat}
\newcommand*{\wt}{\widetilde}

\newcommand*{\cA}{\mathcal{A}}

\newcommand*{\cF}{\mathcal{F}}

\newcommand*{\cH}{\mathcal{H}}

\newcommand*{\cU}{\mathcal{U}}

\newcommand*{\N}{\mathbb{N}}

\newcommand*{\R}{\mathbb{R}}

\DeclareMathOperator*{\essinf}{ess\,inf}
\DeclareMathOperator*{\esssup}{ess\,sup}

\newcommand*{\tr}{\operatorname{tr}}
\newcommand*{\Id}{\operatorname{Id}}

\newcommand*{\Ind}{\mathds{1}}
\newcommand*{\spa}{\operatorname{span}}

\newcommand{\be}{\begin{eqnarray*}}
\newcommand{\ee}{\end{eqnarray*}}
\newcommand{\ben}{\begin{eqnarray}}
\newcommand{\een}{\end{eqnarray}}
\newcommand{\bi}{\begin{itemize}}
\newcommand{\ei}{\end{itemize}}

\newtheorem{theo}{Theorem}[section]

\newtheorem{lemma}[theo]{Lemma}
\newtheorem{propo}[theo]{Proposition}
\newtheorem{corollary}[theo]{Corollary}

\theoremstyle{definition}
\newtheorem{ex}[theo]{Example}
\newtheorem{defi}[theo]{Definition}

\newtheorem{remark}[theo]{Remark}

\title{
Matrix Riccati BSDEs with singular terminal condition and stochastic LQ control with linear terminal constraint
}

\author{Julia Ackermann$^{1}$, Thomas Kruse$^{2}$, Petr Petrov$^{3}$, Alexandre Popier$^{4}$ \bigskip \\
    \small{$^1$ Department of Mathematics \& Informatics, University of Wuppertal,}
    \vspace{-0.1cm}\\
    \small{Germany; e-mail: \texttt{jackermann}\textcircled{\texttt{a}}\texttt{uni-wuppertal.de}}\smallskip\\
    \small{$^2$ Department of Mathematics \& Informatics, University of Wuppertal,}
    \vspace{-0.1cm}\\
    \small{Germany; e-mail: \texttt{tkruse}\textcircled{\texttt{a}}\texttt{uni-wuppertal.de}} \smallskip\\
    \small{$^3$ Department of Mathematics \& Informatics, University of Wuppertal,}
    \vspace{-0.1cm}\\
    \small{Germany; e-mail: \texttt{ppetrov}\textcircled{\texttt{a}}\texttt{uni-wuppertal.de}} \smallskip\\
    \small{$^4$ Laboratoire Manceau de Mathématiques, Le Mans Université,}
    \vspace{-0.1cm}\\
    \small{France; e-mail: \texttt{alexandre.popier}\textcircled{\texttt{a}}\texttt{univ-lemans.fr}} \smallskip\\
}

\begin{document}

\maketitle

\begin{abstract}
    We analyze a class of multidimensional linear-quadratic stochastic control problems with random coefficients, motivated by multi-asset optimal trade execution. The problems feature non-diffusive controlled state dynamics and a terminal constraint that restricts the terminal state to a prescribed random linear subspace. We derive the associated Riccati backward stochastic differential equation (BSDE) and identify a suitable formalization of its singular terminal condition. Via a penalization approach, we establish existence of a minimal supersolution of the Riccati BSDE and use it to characterize both the value function and the optimal control. We analyze the asymptotic behavior of the supersolution near terminal time and discuss special cases where closed-form solutions can be obtained.
\end{abstract}

\section{Introduction}\label{sec:prob_form}

This paper analyzes multidimensional linear-quadratic (LQ) stochastic optimal control problems with random coefficients in which the controlled state process exhibits no diffusive behavior but is required to lie, at the terminal time, in a prescribed random linear subspace of the state space. 

To formally introduce the problem, we fix a dimension $d\in\N$, a time horizon $T\in (0,\infty)$, and a probability space $(\Omega, \cF, P)$ which supports a one-dimensional Brownian motion\footnote{For simplicity we restrict attention to a one-dimensional Brownian motion in the paper. The extension to the multidimensional case does not introduce any essential difficulties and can be handled by analogous arguments.} $W$. Let $(\cF_t)_{t\in [0,T]}$ be the completed filtration generated by~$W$. Let\footnote{Here and in the following we denote the sets of symmetric positive definite, symmetric positive semidefinite  and symmetric $(d \times d)$-matrices by $\mathbb{S}^d_{++}$, $\mathbb{S}^d_{+}$  and $\mathbb{S}^d$, respectively.} $\eta \colon [0,T]\times \Omega \to \mathbb S_{++}^d$, $\lambda  \colon [0,T]\times \Omega \to \mathbb S_+^d$, $\phi  \colon [0,T]\times \Omega \to \mathbb \R^{d \times d}$ and $A \colon [0,T]\times \Omega \to \mathbb \R^{d \times d}$   
be progressively measurable processes, and let $\theta$ be an $\cF_T$-measurable random variable with values in $\mathbb{S}^d_{+}$. 

To formalize the terminal state constraint, we 
let $v_1,\ldots, v_d\colon \Omega \to \R^d$ be $\cF_T$-measurable random vectors 
and denote $C(\omega)=\spa(v_1(\omega),\ldots, v_d(\omega))$ for $\omega \in \Omega$. We do not impose any form of linear independence among the vectors $v_1(\omega),\ldots, v_d(\omega)$, so in particular, the dimension $\dim(C)$ is itself a random quantity.

For all $t\in [0,T]$ and $x\in \R^d$ we consider the problem of minimizing 
\begin{equation}\label{eq:objective}
    J(t,x,u)=E\bigg[\int_t^T  \begin{pmatrix}X^u_s\\ u_s \end{pmatrix}^{\top} \begin{pmatrix} \lambda_s & \phi^{\top}_s\\
         \phi_s & \eta_s
      \end{pmatrix} \begin{pmatrix} X^u_s\\ u_s\end{pmatrix} ds + \langle X^u_T, \theta X^u_T \rangle \, \bigg| \, \cF_t \bigg]
\end{equation} 
over all progressively measurable processes $u\colon [t,T]\times \Omega \to \R^d$ with $u\in L^1([t,T])$ $P$-a.s.\ that induce state processes with the dynamics\footnote{Under appropriate assumptions we can also consider more general state dynamics of the form $\dot{X}_s = B_s u_s + A_s X_s$, $s \in [t, T]$. We refer to Remark \ref{rem:howtoincludeB} for details.} 
\begin{equation}\label{eq:dynamics}
    X^u_s = x + \int_t^s \left(A_r X^u_r - u_r\right) dr,\quad s\in [t,T],
\end{equation}
such that 
\begin{equation}\label{eq:constraint}
    X^u_T = \lim_{s \to T} X^u_s 
    \in C, \quad \text{$P$-a.s.}
\end{equation}
We denote the set of all such $u$ by $\mathcal A_0(t,x)$. 
Moreover, for all $t \in [0,T]$ and $x \in \R^d$, we define 
\begin{equation}\label{eq:val_fct_sing}
    v(t,x)=\essinf_{u\in \mathcal A_0(t,x)}J(t,x,u).
\end{equation}

Since the minimization of \eqref{eq:objective}, together with the state dynamics \eqref{eq:dynamics}, defines an LQ stochastic optimal control problem, albeit with a nonstandard random terminal state constraint \eqref{eq:constraint}, it is natural to seek a characterization of the value function and the optimal strategy via the solution $(Y,Z)$ of a Riccati backward stochastic differential equation (BSDE). The dynamics of $(Y,Z)$ before the terminal time $T$ are standard and can be derived from the literature (see, e.g., \cite{kohl:tang:01,kohl:tang:03}, \cite[Eq. (55)]{sunxiongyong2021indefinite}, 
\cite[Chapter 6]{yong1999stochastic}).

One central question is how to incorporate the terminal constraint \eqref{eq:constraint} into the terminal condition of the Riccati equation. Since, ideally, for every $t\in [0,T)$, the random matrix $Y_t$ characterizes the value function $v(t,\cdot)$ through its associated quadratic form, it is natural to expect that $\lim_{t \to T}\langle \psi_t,Y_t \psi_t\rangle=\infty$ for all paths $(\psi_t)$ with $\lim_{t\to T}\psi_t \notin C$ and $\lim_{t \to T}\langle \psi_t,Y_t \psi_t\rangle=\langle \psi_T,\theta \psi_T\rangle$ for all paths $(\psi_t)$ with $\lim_{t\to T}\psi_t \in C$. 
This together with the specific structure of \eqref{eq:objective} and \eqref{eq:dynamics} suggests the following informal candidate form of the Riccati equation with singular terminal condition 
\begin{equation}\label{eq:BSDE_sing}
\begin{split}
    dY_t&= \bigl((Y_t - \phi^{\top}_t) \eta_t^{-1} ( Y_t - \phi_t) - \lambda_t - Y_t A_t - A^{\top}_t Y_t \bigr)dt+Z_tdW_t,\quad t \in [0,T), \quad\\
    Y_T&=\infty \Ind_{C^c}+\theta \Ind_C.
\end{split}
\end{equation}

Our analysis reveals that, as in the one-dimensional case (see, e.g., \cite{kruse2016minimal} and \cref{rem:discuss_one-dim} below),
it is advantageous to work with supersolutions of \eqref{eq:BSDE_sing} to solve the stochastic control problem \eqref{eq:val_fct_sing}. 
A main contribution of this work is the following rigorous formalization of the singular terminal condition $Y_T=\infty \Ind_{C^c}+\theta \Ind_C$ in \eqref{eq:BSDE_sing}. 

\begin{defi}\label{def:supersol_bsde}
    We call a pair of progressively measurable processes $(Y,Z)\colon [0,T)\times \Omega \to \mathbb{S}^d_+\times \mathbb{S}^d$ a supersolution of \eqref{eq:BSDE_sing} if for all $t\in [0,T)$ it holds that 
    $$E\bigg[ \sup_{s \in [0, t]} \lVert Y_s \rVert^2 + \int_0^t \lVert Z_s \rVert^2 ds \bigg] < \infty$$ 
    and if it holds $P$-a.s.\ for all $t\in [0,T)$ and $s\in [0,t]$ that 
    \begin{equation}\label{eq:supersoln_BSDE_equation}
        Y_t-Y_s= \int_s^t\bigl((Y_r - \phi^{\top}_r) \eta_r^{-1} ( Y_r - \phi_r) - \lambda_r - Y_r A_r - A^{\top}_r Y_r \bigr)dr+\int_s^tZ_rdW_r
    \end{equation}
    and for any process $(\psi_t)_{t \in [0,T]}$ with $P$-a.s.\ left-continuous paths at time $T$ it holds $P$-a.s.\ that: $\liminf_{t\to T} \langle \psi_t, Y_t \psi_t \rangle=\infty$  if $\psi_T \notin C$, and $\liminf_{t\to T} \langle \psi_t, Y_t \psi_t \rangle \ge \langle \psi_T, \theta \psi_T \rangle$  if  $\psi_T \in C$.
\end{defi}

\begin{remark}\label{rem:discuss_one-dim}
    Assume that $d=1$. Then for each $\omega \in \Omega$ we either have $C(\omega)=\{0\}$ or $C(\omega)=\R$. Hence, $Y$ satisfies the singular terminal condition in the sense of \cref{def:supersol_bsde} if and only if $\liminf_{t\to T}Y_t(\omega)=\infty$ if $C(\omega)=\{0\}$ and $\liminf_{t\to T}Y_t(\omega)\ge \theta(\omega)$ if $C(\omega)=\R$. Define the random variable $\xi$ as $\xi(\omega)=\infty$ if $C(\omega)=\{0\}$ and $\xi(\omega)=\theta(\omega)$ if $C(\omega)=\R$. Then the terminal condition can be written as $\liminf_{t\to T}Y_t\ge \xi$. We have thus shown that \cref{def:supersol_bsde} coincides with the definition of supersolutions to BSDEs with singular terminal condition in \cite[Definition 1]{kruse2016minimal} in the case $d=1$.
\end{remark}

The following result summarizes the main findings of this article.
\begin{theo}\label{thm:main_result}
Assume that $A$ is a bounded process, that there exists $\delta \in (0, \infty)$ such that it holds $P$-a.s.\ for all $s \in [0, T]$ that 
    $$\begin{pmatrix} \lambda_s & \phi^{\top}_s\\
         \phi_s & \eta_s
      \end{pmatrix} \ge \begin{pmatrix} 0 & 0\\
         0 & \delta \Id
      \end{pmatrix},$$
and that
$$E\bigg[ \int_0^T \big(  \|\lambda_r\|^3  + \| \phi_r\|^3  + \|\eta_r\|^3 \big) dr\bigg] < \infty.$$ 
Then there exists a minimal solution $(Y,Z)$ to the BSDE \eqref{eq:BSDE_sing} in the sense of \cref{def:supersol_bsde}. Moreover, it holds for all $t\in [0,T)$, $x\in \R^d$ $P$-a.s.\ that $v(t,x)=\langle x,Y_t x \rangle$. 
Finally, for every $x\in \R^d$, $t\in [0,T)$ the optimal state process $(X^*_{s})_{s\in [t,T]}$ satisfies $E[ \int_t^T \lVert \dot{X}^{*}_s\rVert^2 ds ] < \infty$, $X^{*}_s = x +\int_t^s(-\eta_r^{-1}( Y_r - \phi_r) X^{*}_r + A_r X^*_r) dr$, $s \in [t, T]$, and Condition~\eqref{eq:constraint}. 
\end{theo}

\paragraph{Solution approach}
We prove \cref{thm:main_result} via a penalization approach. To this end note first that there exists an $\cF_T$-measurable bounded random matrix $\xi\colon \Omega \to \mathbb{S}^d_+$ 
such that 
\begin{equation}\label{eq:CeqKer}
    C=\ker(\xi) \quad P\text{-a.s.}
\end{equation}
Indeed, for every $\omega \in \Omega$ denote by $C^\perp(\omega)$ the orthogonal complement of $C(\omega)$. Choose an $\mathcal F_T$-measurable orthogonal basis $u_1(\omega),\ldots, u_{k(\omega)}(\omega)$ for $C^\perp(\omega)$. 
Then let $\xi(\omega)=\sum_{i=1}^{k(\omega)}u_i(\omega)u_i^\top(\omega)$ be the orthogonal projection on $C^\perp(\omega)$. It follows that $\xi$ is symmetric positive semidefinite with $C=\ker(\xi)$. Moreover, since $\xi$ is an orthogonal projection it is bounded (in $\omega$). 

Note that $\xi$ satisfying \eqref{eq:CeqKer} is in general not unique. In the sequel we assume that $C$ is given by \eqref{eq:CeqKer} for some $\cF_T$-measurable, symmetric positive semidefinite, bounded random matrix $\xi$. We show below that as expected all results about the solution of the control problem \eqref{eq:val_fct_sing} and the solution of \eqref{eq:BSDE_sing} are independent of the choice of $\xi$ (see in particular \cref{cor:independence_of_xi}).

Based on the construction of $\xi$, we consider a penalized version of \eqref{eq:objective}. More precisely, for all $n\in \N$, $t\in [0,T]$ and $x\in \R^d$ let 
\begin{equation}\label{eq:objective_pen}
    J_n(t,x,u)=E\bigg[\int_t^T \begin{pmatrix}X^u_s\\ u_s \end{pmatrix}^{\top} \begin{pmatrix} \lambda_s & \phi^{\top}_s\\
         \phi_s & \eta_s
      \end{pmatrix} \begin{pmatrix} X^u_s\\ u_s\end{pmatrix} ds +\langle X^u_T , (\theta^n + n \xi) X^u_T \rangle \,\bigg|\, \cF_t \bigg],
\end{equation} 
where $u\colon [t,T]\times \Omega \to \R^d$ is progressively measurable with $u\in L^1([t,T])$ $P$-a.s.\ and $X^u_s= x + \int_t^s \left(A_r X^u_r - u_r\right) dr$, $s \in [t,T]$. We denote the class of all such $u$ by $\mathcal A(t,x)$. Here $\theta^n$, $n\in \N$, is a sequence of bounded random matrices that converges from below to $\theta$ (see \cref{lem:measurable_truncation} and the beginning of \cref{sec:sing_case} for details). Note that we do not impose any terminal state constraint on $u\in \mathcal A(t,x)$. Instead any terminal state with $X^u_T\notin \ker(\xi)=C$ gets increasingly penalized as $n$ tends to $\infty$. Furthermore, for all $n\in\N$, $t \in [0,T]$, and $x \in \R^d$, we define
\begin{equation}\label{eq:value_fct_pen}
    v_n(t,x)=\essinf_{u\in \mathcal A(t,x)}J_n(t,x,u).
\end{equation} 
The Riccati BSDE associated to \eqref{eq:objective_pen} is given by 
\begin{equation}\label{eq:BSDE_pen}
\begin{split}
    dY^n_t&= \bigl((Y^n_t - \phi^{\top}_t) \eta_t^{-1} ( Y^n_t - \phi_t) - \lambda_t - Y^n_t A_t - A^{\top}_t Y^n_t \bigr)dt+Z^n_tdW_t,\quad t \in [0,T],\\
    Y_T^n&=n\xi + \theta^n.
\end{split}
\end{equation}
Recall that $\xi$ is assumed to be bounded. Therefore, for every $n\in \N$ the terminal condition $n\xi + \theta^n$ in \eqref{eq:BSDE_pen} is bounded. 
For this reason we analyze in Section \ref{sec:pen_case} Riccati BSDEs of the form \eqref{eq:BSDE_pen} with generic $\cF_T$-measurable, symmetric positive semidefinite terminal condition $\bar \xi$.

\paragraph{Related literature and applications} Initiated by the work of Wonham \cite{wonham1968matrix} in 1968 and further developed in, e.g., Bismut \cite{bismut1976linear,bismut1978controle} and Davis \cite{davis1977linear} in the 1970s,
LQ stochastic optimal control problems have become one of the most extensively studied topics in stochastic control theory. For a self-contained treatment we refer to \cite{yong1999stochastic} and for more recent references and developments to \cite{sunxiongyong2021indefinite} or \cite{sun2023stochastic}. To the best of our knowledge, however, multidimensional LQ stochastic optimal control problems with a terminal state constraint and the associated 
Riccati BSDE with singular terminal condition have so far only been investigated in \cite{horst2018multidimensionaloptimaltradeexecution}.

As in \cite{horst2018multidimensionaloptimaltradeexecution}, Problem \eqref{eq:val_fct_sing} is inspired by applications in optimal trade execution in mathematical finance. To illustrate this connection, we briefly outline how multi-asset trade execution problems can be formulated as instances of \eqref{eq:val_fct_sing}. In such problems, a typically large institutional investor seeks to liquidate a position $x\in \R^d$ consisting of $d\in \N$ assets within the trading horizon $[t,T]$. The investor controls her position $(X_s)_{s\in [t,T]}$ by adjusting the trading speed $(u_s)_{s\in [t,T]}$ in the $d$ assets so that $dX_s=-u_sds$, $X_t=x$. Here, positive components of $u_s$ correspond to selling the respective asset at the time $s\in [t,T]$, whereas negative components correspond to buying. Departing from the classical assumption that trades can be executed at prevailing market prices, the optimal execution literature incorporates price impact, i.e., the adverse effect of the investors' trades on market prices. If the price impact is assumed to be linear in the trading speed, its magnitude can be modeled by $\eta_s u_s$, leading to instantaneous trading costs $\langle u_s, \eta_s u_s\rangle$ at the time $s\in [t,T]$. Importantly, a non-diagonal matrix process $\eta$ allows for cross-asset effects, reflecting the phenomenon that trading one asset can influence the prices of others. Moreover, by allowing $\eta$ to evolve randomly in time, our model accounts for stochastic liquidity. The quadratic term of the form $\langle X_s, \lambda_s X_s\rangle $ can be interpreted as a risk penalty for holding an open position $X_s$ at the time $s\in [t,T]$. Integrating instantaneous trading costs and risk penalty over time and taking expectations leads to an overall cost functional that is a special case of \eqref{eq:objective}. Finally, the requirement that the portfolio is fully liquidated at the time $T$ gives rise to the terminal state constraint $X_T=0$ (i.e., $C=\{0\}$). We note, however, that our framework also accommodates more general terminal portfolio configurations where the investor aims to achieve a prescribed balance between assets, represented by a linear constraint of the form $\xi X_T=0$.

One-dimensional (i.e., single-asset) variants of \eqref{eq:val_fct_sing} and the associated BSDE \eqref{eq:BSDE_sing} with singular terminal condition have attracted considerable attention in mathematical finance, stochastic control and stochastic analysis in recent years. 
Most of these works go beyond the LQ framework and analyze more general homogeneous cost functionals, relying on BSDE methods, on the theory of superprocesses, or on (viscosity) solutions of the associated Hamilton-Jacobi-Bellman equation, see, e.g., \cite{schied2013control,ankirchner2014bsdes,graewe2018smooth,graewe2015non,ankirchner2015optimal, horst2016constrained, kruse2016minimal,horst2021continuous}. 
The non-homogeneous case is studied in \cite{ankirchner2020optimal}. 
LQ variants incorporating price trends or tracking of a random target position are analyzed in
\cite{ankirchner2015optimal,bank2018linear}.
Well-posedness and continuity properties of the BSDE with singular terminal condition are addressed in \cite{popier2006backward,popier2007backward,popier2016limit,popier2017integro, sezer2019backward,marushkevych2020limit,ahmadi2021backward,grae:popi:21,caci:deni:popi:25, cacitti2025growth,cacitti2025continuity}. Stochastic resilience
is incorporated in \cite{graewe2017optimal,horst2024optimal} and  game-theoretic aspects are examined in \cite{fu2020mean, fu2021mean, fu2024mean}.
The effects of model uncertainty and regime switching dynamics are studied in \cite{horst2022portfolio} and \cite{fu2025system}, respectively.

Alongside these contributions in the one-dimensional situation, a multidimensional LQ formulation is developed in \cite{horst2018multidimensionaloptimaltradeexecution}, where the authors incorporate persistent price impact and stochastic resilience in addition to  deterministic instantaneous cross-price impact.
While the present work neglects persistent price impact and stochastic resilience, it extends \cite{horst2018multidimensionaloptimaltradeexecution} in several important directions: it allows for stochastic instantaneous cross-price impact $\eta$, incorporates a general linear terminal constraint $\xi X_T=0$ and introduces the matrices $A$ in the dynamics \eqref{eq:dynamics} and $\phi$ in the cost functional \eqref{eq:objective}. These extensions give rise to new analytical challenges, including the identification of an appropriate formalization of the singular terminal condition for the Riccati BSDE (cf.\ \eqref{eq:BSDE_sing}) and the derivation of suitable a priori bounds for the (penalized) BSDE (\cref{prop:sol_pen_problem_eta_not_bounded_above} or \cref{prop:lower_bound_without_phi}). Compared to the one-dimensional case, additional technical challenges arise from the lack of monotonicity or boundedness properties of optimal state trajectories.

\paragraph{Outline} 
In Section \ref{sec:pen_case}, we study the control problem \eqref{eq:objective}-\eqref{eq:dynamics} (without the constraint~\eqref{eq:constraint}, that is $C=\mathbb R^d$) and the associated Riccati BSDE \eqref{eq:BSDE_sing} under the convexity condition \ref{assumption:A0}, the boundedness hypothesis \ref{assumption:A1} on $A$, and the third-order moment conditions \ref{assumption:A2} and \ref{assumption:A3}  for the parameters $\lambda$ and $\phi$ and the terminal condition $\theta$. 
This extends 
existing results where typically all coefficients are assumed to be bounded. The main result of \cref{sec:pen_case} is  \cref{prop:sol_pen_problem_eta_not_bounded_above}. At the end of \cref{sec:pen_case}, we also establish several a priori estimates, which will be useful for the constrained case. 

In the next part, Section \ref{sec:sing_case}, we incorporate the constraint \eqref{eq:constraint} on the terminal value of the state process $X$. We obtain the existence of a supersolution for the BSDE with a singular terminal condition and the existence of an optimal control (see \cref{prop:sol_sing_problem}) under Conditions \ref{assumption:A0} to \ref{assumption:A2} and when $\eta$ satisfies the integrability condition \ref{assumption:B2}. Note that there is no condition on the terminal value $\theta$. Roughly speaking, the price to pay is that the solution processes are defined only on $[0,T)$ rather than on the entire interval~$[0,T]$.

The final \cref{sect:invertible_case} is devoted to the case $C=\{0\}$, for which we are able to obtain more precise properties of the BSDE solution and therefore of the optimal control in two scenarios: when $\eta$ has uncorrelated multiplicative increments (\cref{ssect:umi}), or when $\eta$ is an It\^o process (\cref{ssect:eta_ito_process}).

\paragraph{Notation} 
Let $d \in \N$ and $T\in(0,\infty)$. For a matrix $A \in \R^{d\times d}$, we denote by $\lambda_{\max}(A)$ the largest eigenvalue of $A$, by $\lambda_{\min}(A)$ the smallest eigenvalue of $A$, and by $\lVert A \rVert_2$ the spectral norm of $A$. For any matrix $A$, we define by $\lVert A \rVert$ the Frobenius norm of $A$. Moreover, for any matrix $A \in \R^{d\times d}$ we denote by $\tr[A]$ the trace of $A$; note that it holds for all $A, B\in \R^{d\times d}$ that $\tr[AB] \le \lVert A \rVert \lVert B \rVert$. 
 We denote the sets of symmetric positive definite, symmetric positive semidefinite  and symmetric $(d \times d)$-matrices by $\mathbb{S}^d_{++}$, $\mathbb{S}^d_{+}$  and $\mathbb{S}^d$, respectively. For matrices $A, B \in \mathbb{S}^d$ we use the notation $A \ge B$ in the sense that the matrix $(A - B)$ is positive semidefinite. For a sequence of $\mathbb{S}^d$-valued random variables $(Y^n)_{n \in \N}$ on a probability space $(\Omega,\cF,P)$, we say that the sequence is non-decreasing (non-increasing) if for every $n \in \N$ the matrix $(Y^{n+1}-Y^n)$ is $P$-almost surely positive (negative) semidefinite. 
 We use the following space of progressively measurable\footnote{With respect to the filtration generated by the Brownian motion $W$.} processes  for $p\in[1,\infty)$:
 $$ \mathcal{L}^p(\Omega, C([0, T], \mathbb{S}^d))= \{f: \Omega \to  C([0, T], \mathbb{S}^d) : E[\textstyle{\sup_{t \in [0,T]}} \lVert f_t \rVert^p] < \infty \}.$$

\section{The penalized case}\label{sec:pen_case}

In this section we consider the stochastic linear-quadratic control problem~\eqref{eq:value_fct_pen} together with the associated Riccati equation~\eqref{eq:BSDE_pen}. For all $n \in \N$, 
$t \in [0, T]$, and $x \in \R^d$,
this is the problem of minimizing
\begin{equation}\label{eq:objective_pen_sec2}
    \bar{J}(t,x,u)=E\bigg[\int_t^T \begin{pmatrix}X^u_s\\ u_s \end{pmatrix}^{\top} \begin{pmatrix} \lambda_s & \phi^{\top}_s\\
         \phi_s & \eta_s
      \end{pmatrix} \begin{pmatrix} X^u_s\\ u_s\end{pmatrix} ds +\langle X^u_T , \bar{\xi} X^u_T \rangle \,\bigg|\, \cF_t \bigg],
\end{equation}
where $\bar{\xi}$ is an $\cF_T$-measurable random variable with values in $\mathbb{S}^d_+$, the control $u\colon [t,T]\times \Omega \to \R^d$ is progressively measurable with $u\in L^1([t,T])$ $P$-a.s., and the state associated to $u$ is given by $X^u_s=x + \int_t^s \left(A_r X^u_r - u_r\right) dr$, $s \in [t,T]$. As introduced above we denote the set of all such $u$ by $\mathcal A(t,x)$.
Furthermore, for all  $t \in [0,T]$ and $x \in \R^d$, we define
\begin{equation}\label{eq:value_fct_pen2}
    \bar{v}(t,x)=\essinf_{u\in \mathcal A(t,x)}\bar{J}(t,x,u).
\end{equation}
The Riccati equation corresponding to this problem takes the form 
\begin{equation}\label{eq:BSDE_pen2}
    dY_t= \bigl((Y_t - \phi^{\top}_t) \eta_t^{-1} ( Y_t - \phi_t) - \lambda_t - Y_t A_t - A^{\top}_t Y_t \bigr)dt+Z_tdW_t,\quad t \in [0,T], \quad Y_T=\bar{\xi}.
\end{equation}
We now formalize the notion of a solution to this equation. 

\begin{defi}\label{def:sol_bsdes}
     We call a pair of progressively measurable processes $(Y,Z)\colon [0,T]\times \Omega \to \mathbb{S}^d_+\times \mathbb{S}^d$ 
     a solution of \eqref{eq:BSDE_pen2} if $E[ \sup_{s \in [0, T]} \lVert Y_s \rVert^2 + \int_0^T \lVert Z_s \rVert^2 ds ] < \infty$ and if it holds $P$-a.s.\ for all $t\in [0,T]$ that
    \begin{equation}\label{eq:BSDE_finite_term}
        Y_t=\bar{\xi} - \int_t^T \bigl((Y_s - \phi^{\top}_s) \eta_s^{-1} ( Y_s - \phi_s) - \lambda_s - Y_s A_s - A^{\top}_s Y_s \bigr) ds - \int_t^T Z_sdW_s.
    \end{equation}
\end{defi}

Clearly, under appropriate integrability or boundedness assumptions on the coefficients (see, e.g., \ref{assumption:A0}--\ref{assumption:A2} below) the solution component $Y$ has continuous paths.

Next, we present a positivity assumption on the cost matrix. 
\begin{enumerate}[label=\textbf{A\arabic*)}]
    \item\label{assumption:A0} There exists $\delta \in (0, \infty)$ such that it holds $P$-a.s.\ for all $s \in [0, T]$ that 
    $$\begin{pmatrix} \lambda_s & \phi^{\top}_s\\
         \phi_s & \eta_s
      \end{pmatrix} \ge \begin{pmatrix} 0 & 0\\
         0 & \delta \Id
      \end{pmatrix}. $$
\end{enumerate}

\begin{remark}\label{rem:uniform_convexity}
    Note that under Assumption \ref{assumption:A0} the cost functional $J$ is uniformly convex in the sense that for all $t\in [0,T]$ and $u\in \mathcal A(t,0)$ we have
    $$
    \bar J(t,0,u)\ge \delta E\bigg[\int_t^T \|u_s\|^2ds \,\bigg |\,\mathcal F_t \bigg].
    $$
    We use this assumption in the sequel to apply results from \cite{sunxiongyong2021indefinite}. Moreover, this uniform convexity ensures uniqueness of optimal controls $u$ (in the $\text{Leb} \otimes P$-a.s.\ sense).
\end{remark} 

We begin by analyzing the case where~$\eta$, $\lambda$, $\lVert \phi \rVert$, $\bar{\xi}$ and $\|A \|$ are bounded from above, and $\eta$ is bounded away from zero.
Under this assumption, we can draw on existing results in the literature -- as mentioned most notably those of \cite{sunxiongyong2021indefinite} -- to establish existence and uniqueness of a solution to \eqref{eq:BSDE_pen2}, as well as to characterize the value function and the optimal control in \eqref{eq:value_fct_pen2}.

\begin{propo}\label{prop:sol_pen_problem}
    Let assumption \ref{assumption:A0} be satisfied and
    assume that there exist $ K$ and $N$ in $(0, \infty)$ such that it holds $P$-a.s.\ for all $t \in [0, T]$ that $ \eta_t \le N \Id $, $\| A_t \| \le K$, $\| \phi_t\| \le  N$, $ \lambda_t \leq N \Id$, and $0 \leq \bar \xi \leq N\Id$. 
    Then:
    
    (i) There exists a unique solution $(Y,Z)$ of \eqref{eq:BSDE_pen2} in the sense of \cref{def:sol_bsdes}. 
    Furthermore, 
    it holds $P$-a.s.\ for all $t \in [0,T]$ that 
    $$ \lVert Y_t \rVert \le e^{4 K T} E\bigg[ \| \bar{\xi} \| + \int_t^T \| \lambda_s\| ds \,\bigg|\, \cF_t  \bigg].$$
    In particular, it holds $P$-a.s.\ for all $t\in[0,T]$ that 
    \begin{equation*}
        \lVert Y_t \rVert 
        \le e^{4 K T} (1+T) N \sqrt{d} .
    \end{equation*}
    
    (ii) For all $t\in [0,T]$ and $x\in \R^d$ it holds $P$-a.s.\ that $\bar{v}(t,x)=\langle x,Y_t x\rangle$.
    
    (iii) For all $t\in [0,T]$ and $x\in \R^d$ the optimal state process $X^*$ satisfies $E[ \int_t^T \lVert \dot{X}^{*}_s\rVert^2 ds ] < \infty$ and $X^{*}_s = x +\int_t^s(-\eta_r^{-1}( Y_r - \phi_r) X^{*}_r + A_r X^*_r) dr$, $s \in [t, T]$.
\end{propo}

\begin{proof}
    The existence and uniqueness of the BSDE solution $(Y,Z)$ follows from \cite[Theorem~6.1]{sunxiongyong2021indefinite} (note also Remark \ref{rem:uniform_convexity}). 
    It also follows from \cite{sunxiongyong2021indefinite} for all $t \in [0, T]$ and $x \in \R^d$ that 
    $$\langle x, Y_t x\rangle = \essinf_{u \in \cU(t, x)} \bar{J}(t, x, u),$$  
    where $\cU(t,x)=\mathcal A(t,x) \cap L^2(\Omega \times [t, T])$. Note that $\mathcal{U}(t,x) \subset \mathcal A(t,x)$, hence it holds $$\essinf_{u \in \cU(t, x)} \bar{J}(t, x, u) \geq \essinf_{u \in \cA(t, x)} \bar{J}(t, x, u).$$
    If there is a control $u \in \mathcal A(t,x)$ such that $\bar{J}(t,x,u) \leq \langle x, Y_t x\rangle $, then since $Y_t \in L^2(\Omega)$, we deduce using assumption \ref{assumption:A0} and positive semidefiniteness of $\bar{\xi}$ that 
    $$\delta E\bigg[ \int_t^T \langle u_r, u _r \rangle dr\bigg] 
    \le E\big[ \bar{J}(t,x,u) \big] 
    \leq E\big[ \langle x, Y_t x\rangle\big] < \infty.$$ 
    Hence $u \in \cU(t,x)$. Thus, we obtain $\langle x, Y_t x\rangle = \bar{v} (t,x)$. 
    This also implies, for all $t\in[0,T]$, that $Y_t\ge 0$. 

    To obtain the upper bound, let $(U_t)_{t\in[0,T]}$ be given by 
    $$d U_t = - U_t A_t dt, \quad t \in [0,T], \quad U_0 = \Id.$$
    Note that, $P$-a.s.\ for all $t \in [0,T]$, the matrix $U_t$ is invertible, and the inverse satisfies the dynamics
    $$ dU^{-1}_t = A_t U^{-1}_t dt, \quad t \in [0,T], \quad U^{-1}_0 = \Id.$$
    Moreover, let $M=(M_t)_{t\in[0,T]}$ be given by $M_t=(U_t^{-1})^\top Y_t U_t^{-1}$, $t\in[0,T]$. 
    We obtain by integration by parts that $M$ satisfies 
    \begin{align}\nonumber
       dM_t &=  (U_t^\top)^{-1}\bigl((Y_t - \phi^{\top}_t) \eta_t^{-1} ( Y_t - \phi_t) - \lambda_t  \bigr)U_t^{-1} dt+(U_t^\top)^{-1}Z_tU_t^{-1}dW_t, \quad t \in [0,T],\\ \label{eq:def_M_for_upper_bound}
       M_T &= (U^\top_T)^{-1} \bar{\xi} U_T^{-1}.  
    \end{align}
    Note that, for all $t \in [0,T]$, it holds that $ \| U_t \| \le e^{KT} $ and $ \| U_t^{-1} \| \le e^{KT} $. Furthermore, $ (U^\top)^{-1}((Y - \phi^{\top}) \eta^{-1} ( Y - \phi)  )U^{-1}$ is nonnegative, since $\eta^{-1}$ is. Hence, we have from \eqref{eq:def_M_for_upper_bound} the upper bound 
    $$ \lVert M_t \rVert \le e^{2 K T} E\bigg[ \| \bar{\xi} \| + \int_t^T \| \lambda_s\| ds \,\bigg|\, \cF_t  \bigg], \quad t\in[0,T]. $$
    From this, we immediately obtain for $Y$ the upper bound 
    $$ \lVert Y_t \rVert \le e^{4 K T} E\bigg[ \| \bar{\xi} \| + \int_t^T \| \lambda_s\| ds \,\bigg|\, \cF_t  \bigg], \quad t \in [0,T] .$$
    
    Furthermore, it follows from \cite[Theorem~6.1]{sunxiongyong2021indefinite} that for all $t\in [0,T]$ and $x\in \R^d$ there exists a unique optimal control $u^*$. Moreover, the optimal state process $X^*=X^{u^*}$ satisfies $X^{*}_s = x +\int_t^s(-\eta_r^{-1}( Y_r - \phi_r) X^{*}_r + A_r X^*_r) dr$, $s \in [t, T]$, and the optimal control is given by $u^{*}_s = \eta_s^{-1}( Y_s - \phi_s) X^{*}_r$, $s \in [t, T]$.
\end{proof}

The next result provides a way to construct monotone truncations of progressively measurable processes taking values in $\mathbb{S}^d_+$ (or $\mathbb{S}^d_{++}$) such that the truncations preserve symmetry, measurability, and positive (semi-)definiteness.

\begin{lemma}\label{lem:measurable_truncation}
    Let $f: [0,T]\times \Omega \to \mathbb{S}^d_+$ be an $(\cF_t)$-progressively measurable process. 
    For all $\omega \in \Omega$ and $s \in [0,T]$ let $\bar{f}_s(\omega) = \max_{i,j \in \{1, \dots, d\}} \lvert (f_{i,j})_s(\omega) \rvert$. For all $\omega \in \Omega$, $L \in \N$, and $s \in [0,T]$ we define 
$$
f^L_s(\omega) =
\begin{cases}
 \left( 1 \wedge L (\bar{f}_s(\omega))^{-1} \right) f_s(\omega), & \text{if } \bar{f}_s(\omega) \neq 0, \\
0,    & \text{otherwise}.
\end{cases} $$
Then the process $f^L: [0,T]\times \Omega \to \mathbb{S}^d_+$ is $(\cF_t)$-progressively measurable and uniformly bounded for any $L \in \N$. Moreover, the sequence $(f^L_s(\omega))_{L \in \N}$ is monotone for all $\omega \in \Omega$ and $s \in [0,T]$. 
\end{lemma}

\begin{proof}
All properties directly follow from the construction.
\end{proof}

Our next aim is to weaken the assumption in \cref{prop:sol_pen_problem} that $\eta$, $\lambda$, $\lVert \phi \rVert$ and $\bar{\xi}$ are bounded from above. Compared to the one-dimensional case, the monotonicity argument (see, e.g., \cite[Section 3]{ankirchner2014bsdes}) cannot be applied here; instead, we impose integrability conditions on the parameters. 
\begin{enumerate}[label=\textbf{A\arabic*)}]
\setcounter{enumi}{1}
    \item\label{assumption:A1} 
    There exists $K \in (0, \infty)$ such that it holds $P$-a.s.\ for all $t \in [0, T]$ that $\| A_t\| \le K $.
    
    \item\label{assumption:A2}  
    The processes $\lambda$ and $\phi$ have a moment of order three, that is, 
        $$E\bigg[ 
        \int_0^T (\|\lambda_r\|^3  + \| \phi_r\|^3 )\, dr\bigg] < \infty.$$
  
    \item\label{assumption:A3} The terminal condition $\bar \xi$ has a third moment, that is, $E[  \|\bar{\xi}\|^3 ] < \infty$. 
\end{enumerate}

\begin{lemma}\label{lem:continuous_state_sec2}
    Assume that \ref{assumption:A0}--\ref{assumption:A3} are satisfied and that there exists a solution $(Y,Z)$ of the BSDE~\eqref{eq:BSDE_pen2} in the sense of \cref{def:sol_bsdes}. 
    Let $t\in[0,T]$ and $x\in\R^d$. 
    Let $u=(u_s)_{s\in[t,T]}$ be given by $u_s=\eta_s^{-1} (Y_s - \phi_s) X_s$, $s\in[t,T]$, where $(X_s)_{s\in[t,T]}$ is the unique solution of the differential equation with random coefficients
    \begin{equation}\label{eq:ODE_with_random_coeffs_for_X}
        dX_s= (- \eta_s^{-1} (Y_s-\phi_s) + A_s) X_s ds, \quad s\in[t,T], \quad X_t=x.
    \end{equation}
    Then it holds that $E[\int_t^T \lVert u_s \rVert^2 ds]<\infty$ and $\langle x, Y_t x \rangle \ge \bar{J}(t,x,u)$.
\end{lemma}

\begin{proof}
First note that since $A$ and $\eta^{-1}$ are bounded and $Y$ and $\phi$ are in $L^1([0,T]\times \Omega)$, there indeed exists a unique solution of \eqref{eq:ODE_with_random_coeffs_for_X} 
(see, for example, \cite[Theorem~3.1.1]{liu2015spde}). 
It holds almost surely that the paths of  
$(X_s)_{s\in [t,T]}$ are continuous and hence bounded and that $Z$ is in $L^2([0,T] \times \Omega)$. Hence $( \int_t^r \langle X_s,Z_sX_s  \rangle dW_s)_{r\in[t,T]}$ is a local martingale 
and there exists a sequence of stopping times $(\tau_k)_{k\in\N}$ that is increasing, converges to $T$, and makes this local martingale stopped at the time $\tau_k$ a true martingale for all $k\in\N$. 
It\^o's formula implies for all $k\in\N$ that 
\begin{align*} 
            & \langle X_{\tau_k}, Y_{\tau_k} X_{\tau_k}\rangle - \langle x , Y_t x \rangle\\
            &=
            \int_t^{\tau_k} \big\langle Y_s X_s, \big(-\eta_s^{-1}(Y_s-\phi_s)+A_s \big) X_s \big\rangle ds 
            + \int_t^{\tau_k} \langle X_s, Z_sX_s\rangle dW_s \\  
            & \quad + \int_t^{\tau_k} \big\langle X_s, \big( (Y_s-\phi_s^\top) \eta_s^{-1} (Y_s-\phi_s) - \lambda_s - Y_s A_s - A_s^\top Y_s \big) X_s \big\rangle ds \\  
            & \quad + \int_t^{\tau_k} \big\langle X_s, \big( - (Y_s-\phi_s^\top) \eta_s^{-1} + A_s^\top \big) Y_s X_s \big\rangle ds 
            \\  
            &= - \int_t^{\tau_k} \langle X_s, \lambda_s X_s \rangle ds  
            + \int_t^{\tau_k}\langle u_s, \eta_s u_s \rangle ds 
            - 2 \int_t^{\tau_k} \langle Y_s X_s, u_s \rangle ds
            + \int_t^{\tau_k} \langle X_s,Z_sX_s \rangle dW_s\\
            &= - \int_t^{\tau_k} \big( \langle X_s, \lambda_s X_s \rangle  
            + \langle u_s, \eta_s u_s \rangle 
            + 2 \langle \phi_s X_s, u_s \rangle \big) ds
            + \int_t^{\tau_k} \langle X_s,Z_sX_s \rangle dW_s .
\end{align*}
Taking conditional expectations, we obtain for all $k\in\N$ that 
\begin{equation}\label{eq:1800a}
   E \bigg[  \int_t^{ \tau_k} \big(\langle u_s, \eta_s u_s \rangle + 2 \langle \phi_s X_s, u_s \rangle + \langle X_s, \lambda_s X_s \rangle \big) ds + \langle X_{ \tau_k}, Y_{\tau_k} X_{ \tau_k} \rangle \, \bigg|\, \mathcal F_t \bigg] =  \langle x, Y_t x\rangle. 
\end{equation}
Observe that by \ref{assumption:A0} and the monotone convergence theorem we have 
$$ \delta E \bigg[  \int_t^{T} \| u_s \|^2 ds\bigg] \le E\bigg[ \langle x, Y_t x \rangle \bigg] < \infty.  $$ 
This shows that 
$E[\int_t^T \| u_s \|^2 ds] < \infty$. 
Combining \ref{assumption:A0}, the monotone convergence theorem, and Fatou's lemma with  \eqref{eq:1800a}, we obtain that 
\begin{equation*}\label{eq:optim_control_cost}
    \begin{split}
    \langle x, Y_t x\rangle 
    &= \liminf_{k \to \infty} E\bigg[\int_t^{\tau_k} \big( \langle u_s, \eta_s u_s \rangle + 2 \langle \phi_s X_s, u_s \rangle + \langle X_s, \lambda_s X_s \rangle \big) ds + \langle X_{\tau_k }, Y_{\tau_k } X_{\tau_k} \rangle \, \bigg|\, \cF_t \bigg]\\ 
    &  \geq E\bigg[  \int_t^{ T} \big(\langle u_s, \eta_s u_s \rangle + 2 \langle \phi_s X_s, u_s \rangle+ \langle X_s, \lambda_s X_s \rangle \big) ds + \liminf_{k \to \infty} \langle X_{ \tau_k}, Y_{\tau_k} X_{ \tau_k} \rangle \, \bigg|\, \mathcal F_t \bigg] \\
    & = E\bigg[  \int_t^{ T} \big(\langle u_s, \eta_s u_s \rangle+2 \langle \phi_s X_s, u_s \rangle + \langle X_s, \lambda_s X_s \rangle \big) ds +  \langle X_{T}, \bar{\xi} X_{T} \rangle \,\bigg|\, \mathcal F_t \bigg] \\
    & = \bar{J}(t,x,u). 
    \end{split}
\end{equation*}
This completes the proof.
\end{proof}

\begin{propo}\label{prop:sol_pen_problem_eta_not_bounded_above}
    Assume that \ref{assumption:A0}--\ref{assumption:A3} are satisfied.  Then:
    
    (i) There exists a solution $(Y,Z)$ of \eqref{eq:BSDE_pen2} in the sense of \cref{def:sol_bsdes}. Furthermore, it holds $P$-a.s.\ for all $t \in [0,T]$ that 
    \begin{equation}\label{eq:upper_bound_Y}
        \| Y_t \| \leq e^{4 K T} E\bigg[  \|\bar{\xi}\| + \int_t^T \|\lambda_r \| dr  \,\bigg|\, \mathcal F_t \bigg].
    \end{equation}
    In particular, $Y \in L^3 (\Omega \times [0,T])$.
    
    (ii) For all $t\in [0,T]$ and $x\in \R^d$ it holds $P$-a.s.\ that $\bar{v}(t,x)=\langle x,Y_t x\rangle$.
    
    (iii) For all $t\in [0,T]$ and $x\in \R^d$ the optimal state process $X^*$ satisfies $E[ \int_t^T \lVert \dot{X}^{*}_s\rVert^2 ds ] < \infty$ and $\dot{X}^{*}_s = - \eta_s^{-1}( Y_s - \phi_s) X^{*}_s + A_s X^*_s$, $s \in [t, T].$
\end{propo}

\begin{proof}
The idea of the proof is to construct a solution to the BSDE as a limit of solutions to BSDEs with bounded coefficients. The solution to a BSDE with bounded coefficients is given in \cref{prop:sol_pen_problem}. After performing truncations given by \cref{lem:measurable_truncation} we want to obtain a set of ``monotone'' problems satisfying the lower bound \ref{assumption:A0}, for this reason we introduce auxiliary processes. 
For all $s\in[0,T]$ set 
$\bar{\eta}_s = {\eta_s}-\delta \Id$ and
\begin{equation*}
    G_s = \begin{pmatrix}
    \lambda_s & \phi^{\top}_s\\
         \phi_s & \bar{\eta}_s
      \end{pmatrix} .
\end{equation*} 
Note that the process $G=(G_s)_{s \in [0,T]}$ satisfies the conditions of \cref{lem:measurable_truncation} 
(with the dimension $2d$ instead of $d$). 
For all $L\in\N$ let $G^L=(G^L_s)_{s\in [0,T]}$ be the process obtained by truncating $G$ as introduced in \cref{lem:measurable_truncation}. We denote for all $L\in\N$ the blocks of $G^L$ corresponding to $\lambda$, $\phi$, and $\bar{\eta}=(\bar{\eta}_s)_{s\in[0,T]}$ by $\lambda^L$, $\phi^L$, and $\bar{\eta}^L$, respectively. 
Note that the processes $\lambda^L$, $\phi^L$, and $\bar{\eta}^L$ are uniformly bounded for any $L\in\N$. 
Moreover, for all $L\in\N$ and $s\in[0,T]$ we define $\eta^L_s  = \bar{\eta}^L_s + \delta \Id$. 
In addition, we apply a version of \cref{lem:measurable_truncation} for $\mathbb{S}^d_+$-valued random variables to $\bar{\xi}$ and denote the resulting sequence of truncated random variables by $(\bar{\xi}^L)_{L\in\N}$.

Now, \cref{prop:sol_pen_problem} establishes that for all $L\in\N$ there exists a unique solution $(Y^{ L}, Z^{ L})$ of the BSDE 
$$ dY^{ L}_s = ((Y^{ L}_s - (\phi^L_s)^\top) (\eta^L_s)^{-1} ( Y^{ L}_s - \phi^L_s) - \lambda^L_s - Y^L_s A_s - A^\top_s Y^L_s)ds+ Z^{ L}_s dW_s, \quad s\in[0,T], $$
with the terminal condition $ Y^{L}_T =  \bar{\xi}^L$. 
The value function representation (cf.\ \cref{prop:sol_pen_problem}(ii)) implies for all $x \in \R^d$, $t\in[0,T]$, and $L\in\N$ that 
\begin{equation}\label{eq:value_fct_rep_Y_nl}
    \langle x, Y^{ L}_t x \rangle = \essinf_{u \in \cA(t,x)} E\bigg[\int_t^T \begin{pmatrix}X^u_s\\ u_s \end{pmatrix}^{\top} 
    \begin{pmatrix} \lambda^L_s & (\phi^{L}_s)^\top\\
         \phi^L_s & \eta^L_s
      \end{pmatrix} 
      \begin{pmatrix} X^u_s\\ u_s\end{pmatrix} 
      ds 
      +  \langle X^u_T, \bar{\xi}^L X^u_T\rangle \, \bigg|\,\cF_t\bigg].
\end{equation}
The fact that for all $s \in [0,T]$ the sequences $(G^L_s)_{L\in\N}$ and $(\bar{\xi}^L)_{L\in\N}$ are non-decreasing 
ensures that for all $s \in [0,T]$ the sequence $(Y^{L}_s)_{L \in \N}$ is non-decreasing. 
Moreover, it holds $P$-a.s.\ for all $t\in[0,T]$ and $L\in\N$ that 
\begin{equation*}
    \| Y^{L}_t \| \le e^{4KT} E\bigg[  \| \bar{\xi}^L \| + \int_t^T \| \lambda^L_s \| ds \,\Big|\, \mathcal F_t \bigg] 
\leq e^{4KT} E\bigg[  \| \bar{\xi} \| + \int_t^T \| \lambda_s \| ds \,\Big|\, \mathcal F_t \bigg].
\end{equation*}
Using \cref{lemma:matrix_convergence} 
we can, for all $s\in[0,T]$, define $Y_s$ as the $P$-a.s.\ limit of $Y^{ L}_s$ as $L \to \infty$. From this definition it immediately follows 
that the stochastic process $Y=(Y_s)_{s\in[0,T]}$ is $P$-a.s.\ $\mathbb{S}_+^d$-valued and satisfies the upper bound \eqref{eq:upper_bound_Y}. 
Moreover, note that \eqref{eq:upper_bound_Y}, Jensen's inequality, \ref{assumption:A2}, and  \ref{assumption:A3} ensure that $Y \in L^3(\Omega \times [0,T] )$.

We next show that $Y=(Y_s)_{s\in[0,T]}$ solves the BSDE \eqref{eq:BSDE_pen2}. 
For all 
$M, L \in \N$, 
applying It\^o's formula gives for all $t\in[0,T]$ that
\begin{equation}\label{eq:1650a}
    \begin{split}
    & \lVert Y^{ M}_t - Y^{ L}_t \rVert^2 + \int_t^T \lVert Z^{ {M}}_s - Z^{ {L}}_s \rVert^2 ds \\
    & = \| \bar{\xi}^M - \bar{\xi}^L \|^2 - 2 \int_t^T \tr \left[(Y^{ M}_s - Y^{ L}_s)(Z^{ M}_s - Z^{ L}_s) \right] d W_s  \\
    & \quad + 2 \int_t^T \tr \big[(Y^{ M}_s - Y^{ L}_s)(\lambda^M_s- \lambda^L_s) \big] ds \\
    & \quad + 2 \int_t^T \tr \big[(Y^{ M}_s - Y^{ L}_s)(Y^M_s A_s + A_s^\top Y^M_s - Y^L_s A_s - A_s^\top Y^L_s) \big] ds \\
    & \quad + 2 \int_t^T \tr \Big[(Y^{ M}_s - Y^{ L}_s)\Big(-(Y^{ M}_s - (\phi^M_s)^{\top})(\eta^M_s )^{-1}( Y^{ M}_s - \phi^M_s)
    \\
    &\qquad \qquad \qquad \qquad \qquad \qquad+ (Y^{ L}_s - (\phi^L_s)^{\top})(\eta^L_s)^{-1}( Y^{ L}_s - \phi^L_s)\Big) \Big] ds .
    \end{split}
\end{equation}
Note that for all $L\in\N$ we have by \cref{prop:sol_pen_problem} that there exists a constant $C_L$ such that for all $t \in [0,T]$ it holds $\lVert Y_t^{ L}\rVert \le C_L $ 
and $E [\int_0^T \lVert Z^{ L}_s \rVert^2 ds] < \infty $.
Combining these facts yields for all $M, L\in\N$ and $t\in[0,T]$ that
\begin{equation}\label{eq:1650b}
    E\left[\int_t^T \tr \left[(Y^{ M}_s - Y^{ L}_s)(Z^{ M}_s - Z^{ L}_s) \right] d W_s\right] = 0 .
\end{equation}
Observe that for all $L \in \N$ and $s\in[0,T]$ it holds $P$-a.s.\ that $\lVert (\eta^L_s)^{-1} \rVert \le \sqrt{d} \delta^{-1}$. 
It follows  
for all $M,L \in \N$ and $s\in[0,T]$ that
\begin{align} \nonumber
    &  \tr \!\Big[ \!(Y^{ M}_s - Y^{ L}_s)\big(-(Y^{ M}_s  - (\phi^M_s)^{\top})(\eta^M_s )^{-1}( Y^{ M}_s - \phi^M_s) 
    + (Y^{ L}_s - (\phi^L_s)^{\top})(\eta^L_s)^{-1}( Y^{ L}_s - \phi^L_s)\big) \! \Big] \\ \nonumber
    & \le d^{\frac{1}{2}} \delta^{-1} \| Y^M_s - Y^L_s \| \big(  \| Y^M_s\|^2 \! + \! \| Y^L_s \|^2 + 2 (\|Y^M_s\|\|\phi^M_s\| +\|Y^L_s\|\|\phi^L_s\|) 
    +\|\phi^M_s\|^2 \! + \! \|\phi^L_s\|^2  \big) 
    \\
    & \le 2d^{\frac{1}{2}} \delta^{-1} \| Y^M_s - Y_s^L \| \left( \| Y_s\|^2  + 2 \| Y_s\| \| \phi_s \| + \| \phi_s \|^2 \right) . \label{eq:traceinequ1a}
\end{align} 
Moreover, we obtain for all $M,L \in \N$ and $s\in[0,T]$ that 
\begin{equation}\label{eq:traceinequ1b}
    \begin{split}
        & \tr \big[(Y^{ M}_s - Y^{ L}_s)(\lambda^M_s- \lambda^L_s) \big]  
        \le 2 \lVert Y^{ M}_s - Y^{ L}_s \rVert \lVert \lambda_s \rVert .
    \end{split}
\end{equation}
Furthermore, \ref{assumption:A1} ensures for all $M,L \in \N$ and $s\in[0,T]$ that 
\begin{equation}\label{eq:traceinequ1c}
    \begin{split}
        & \tr \big[(Y^{ M}_s - Y^{ L}_s)(Y^M_s A_s + A_s^\top Y^M_s - Y^L_s A_s - A_s^\top Y^L_s) \big] 
        \le 4K \lVert Y_s \rVert \lVert Y_s^M - Y_s^L \rVert
        .
    \end{split}
\end{equation}
For all 
$s\in[0,T]$ we define 
\begin{equation}\label{eq:1650cc}
    \begin{split}
        R_s
        & = \big( 2 d^{\frac{1}{2}}  \delta^{-1} (  \| Y_s\|^2+  2 \|Y_s\| \|\phi_s\| + \|\phi_s\|^2 ) 
        + 2 \lVert \lambda_s \rVert + 4 K \| Y_s \| \big) .
    \end{split}
\end{equation}
Note that it holds for all $L,M \in \N$ and $s\in[0,T]$ that $R_s \lVert Y^{ M}_s - Y^{ L}_s \rVert \le 2 R_s \lVert Y_s \rVert$. 
From Hölder's inequality, $Y \in L^3(\Omega \times [0,T] )$, and \ref{assumption:A2} we have that 
$E[\int_0^T R_s \lVert Y_s \rVert\, ds ]<\infty$. 
Therefore, the dominated convergence theorem proves that 
\begin{equation}\label{eq:integralRLMconverges}
    \lim_{L,M\to\infty} E\bigg[ \int_0^T R_s \lVert Y^{ M}_s - Y^{ L}_s \rVert \, ds \bigg] = 0.
\end{equation}
In addition, note that 
\begin{equation}\label{eq:xibarconvergesinL2}
    \lim_{L,M \to\infty} E\big[\| \bar{\xi}^M - \bar{\xi}^L \|^2\big] = 0
\end{equation}
due to \ref{assumption:A3} and the dominated convergence theorem. 
Observe that \eqref{eq:1650a}--\eqref{eq:1650cc} 
show for all $M,L\in\N$ that 
\begin{equation}\label{eq:estimateZconvergesinL2}
    \begin{split}
        E\bigg[ \int_0^T \lVert Z^{M}_s - Z^{ L}_s \rVert^2 ds \bigg] 
        & \le E\left[ \|\bar{\xi}^M - \bar{\xi}^L \|^2 \right] 
        + 2 E\bigg[ \int_0^T R_s \lVert Y^{ M}_s - Y^{ L}_s \rVert \, ds \bigg] .
    \end{split}
\end{equation} 
This combined with \eqref{eq:integralRLMconverges} and \eqref{eq:xibarconvergesinL2} implies
that $(Z^{ L})_{L\in\N}$ is a Cauchy sequence in the space $L^2(\Omega \times [0,T])$ and thus converges to some $Z \in L^2(\Omega \times [0,T])$. 
Next, \eqref{eq:1650a} and \eqref{eq:1650cc} show for all $M,L\in\N$ that  
\begin{equation}\label{eq:supYMLinequ}
    \begin{split}
    \sup_{t \in [0,T]} \lVert Y^{ M}_t - Y^{ L}_t \rVert^2 
    & \leq 
    \| \bar{\xi}^M - \bar{\xi}^L \|^2 + 
    2 \sup_{t \in [0,T]}\left| \int_t^T \tr \left[(Y^{ M}_s - Y^{ L}_s)(Z^{ M}_s - Z^{ L}_s) \right] d W_s \right| \\
    & \quad + 2 \int_0^T  R_s \lVert Y^{ M}_s - Y^{ L}_s \rVert \, ds .
    \end{split}
\end{equation}
Using the Burkholder--Davis--Gundy inequality and  Young's inequality, we obtain that there exists a constant $c>0$ such that 
for all $M,L\in\N$ it holds 
\begin{equation*}
    \begin{split}
        & E\bigg[ \sup_{t \in [0,T]}\left| \int_t^T \tr \left[(Y^{ M}_s - Y^{ L}_s)(Z^{ M}_s - Z^{ L}_s) \right] d W_s \right| \bigg] \\
        & \leq c E\bigg[ \Big(\sup_{t\in[0,T]} \lVert Y^M_t - Y^L_t \rVert \Big) \bigg(\int_0^T \lVert Z^M_s - Z^L_s \rVert^2 \, ds\bigg)^{\frac{1}{2}} \bigg] \\
        & \le \frac{1}{4} E\bigg[ \sup_{t\in[0,T]} \lVert Y^M_t - Y^L_t \rVert^2 \bigg] 
        + c^2 E\bigg[ \int_0^T \lVert Z^M_s - Z^L_s \rVert^2 \, ds \bigg] .
    \end{split}
\end{equation*}
It thus follows from \eqref{eq:supYMLinequ} for all $M,L\in\N$ that 
\begin{equation*}
    \begin{split}
    \frac12 E\bigg[\sup_{t \in [0,T]} \lVert Y^{ M}_t - Y^{ L}_t \rVert^2 \bigg]
    & \leq 
    E\big[ \| \bar{\xi}^M - \bar{\xi}^L \|^2 \big] 
    + 2 c^2 E\bigg[ \int_0^T \lVert Z^M_s - Z^L_s \rVert^2 \, ds \bigg] \\
    & \quad + 2 E\bigg[\int_0^T  R_s \lVert Y^{ M}_s - Y^{ L}_s \rVert \, ds \bigg] .
    \end{split}
\end{equation*}
Hence, using \eqref{eq:estimateZconvergesinL2}, \eqref{eq:xibarconvergesinL2}, and \eqref{eq:integralRLMconverges}, we obtain that 
$(Y^{L})_{L \in \mathbb N}$ is also a Cauchy sequence in the space $\mathcal{L}^2(\Omega, C([0, T], \mathbb{S}^d))$. 
Furthermore, note that by using the dominated convergence theorem we can establish for all $t\in[0,T]$ that 
\begin{equation*}
   \lim_{L\to\infty} E\bigg[ \Big\lVert \int_t^T (Y^{ L}_s - (\phi^L_s)^{\top})(\eta^L_s)^{-1}( Y^{ L}_s - \phi^L_s) ds - \int_t^T (Y_s - \phi_s^{\top})(\eta_s)^{-1}( Y_s - \phi_s) ds \Big\rVert \bigg] = 0.
\end{equation*}
Finally, we obtain for all $t\in[0,T]$ that 
\begin{equation*}
   \begin{split}
   & E\bigg[ \bigg\lVert Y_t 
   - \bigg( \bar{\xi} 
    - \int_t^T (Y_s - \phi_s^{\top})(\eta_s)^{-1}( Y_s - \phi_s) ds + \int_t^T (\lambda_s + Y_s A_s + A_s^{\top} Y_s) ds
    \\
    & \qquad \qquad \quad - \int_t^T Z_s dW_s \bigg) 
    \bigg\rVert \bigg] \\
    & \le \lim_{L\to \infty} 
    \bigg( E\big[ \lVert Y_t - Y^{ L}_t \rVert \big]
   + E\bigg[ \Big\lVert \int_t^T Z_s dW_s - \int_t^T Z^{ L}_s dW_s \Big\rVert \bigg] + E\big[ \lVert \bar{\xi} - \bar{\xi}^{ L} \rVert \big]  \\
   & \qquad \qquad + E\bigg[ \Big\lVert \int_t^T (Y^{ L}_s - (\phi^L_s)^{\top})(\eta^L_s)^{-1}( Y^{ L}_s - \phi^L_s) ds \\
   & \qquad \qquad \qquad \qquad- \int_t^T (Y_s - \phi_s^{\top})(\eta_s)^{-1}( Y_s - \phi_s) ds \Big\rVert \bigg]\\
   & \qquad \qquad + E\bigg[ \Big\lVert \int_t^T (\lambda_s + Y_s A_s + A_s^{\top} Y_s) ds - \int_t^T (\lambda^L_s +  Y^L_s A_s + A_s^{\top} Y^L_s) ds \Big\rVert \bigg]  \bigg) 
   = 0.
   \end{split}
\end{equation*} 
It follows for all $t\in[0,T]$ that $P$-a.s.\ 
\begin{equation}\label{eq:1800}
    Y_t = \bar{\xi}
    + \int_t^T \left( - (Y_s - \phi_s^{\top})\eta_s^{-1}( Y_s - \phi_s)  + \lambda_s + Y_s A_s + A_s^\top Y_s \right) ds 
    - \int_t^T Z_s dW_s .
\end{equation}
This yields that the process on the left-hand side of \eqref{eq:1800} is a modification of the process on the right-hand side.  
Next, we show that $\eqref{eq:1800}$ holds $P$-a.s.\ for all $t \in [0,T]$, i.e., these two processes are indistinguishable. Since $Y^L$ is $P$-a.s.\ continuous for any $L \in \N$ and it holds $\lim_{L\to\infty}E[\sup_{t \in [0,T]} \| Y^L_t - Y_t \|^2] = 0$, we obtain that the process $Y$ is $P$-a.s.\ continuous. Moreover, remark that the processes $(\int_0^t ( - (Y_s - \phi_s^{\top})\eta_s^{-1}( Y_s - \phi_s)  + \lambda_s + Y_s A_s + A_s^\top Y_s) ds)_{t \in [0, T]} $ and $(\int_0^t Z_s dW_s)_{t \in [0, T]}$ are $P$-a.s.\ continuous. This implies that \eqref{eq:1800} holds $P$-a.s.\ for all $t \in [0,T]$, and the pair $(Y,Z)$ solves the BSDE \eqref{eq:BSDE_pen2} in the sense of \cref{def:sol_bsdes}. This proves (i).

To prove (ii) and (iii), fix $t\in [0,T]$ and $x\in \R^d$. The fact that for all $L\in\N$ and $s\in[0,T]$ we have $G^L_s\le G_s$ and $\bar{\xi}^L \leq \bar{\xi}$ together with \eqref{eq:value_fct_rep_Y_nl} implies for all $L\in\N$ that
$\langle x, Y^{ L}_t x \rangle \le \bar{v}(t,x)$.
Taking the limit $L \to \infty$, we obtain
$\langle x, Y_t x \rangle  \le \bar{v}(t,x)$.  
Combining this and \cref{lem:continuous_state_sec2} yields 
$ \bar{v}(t,x) =  \langle x, Y_t x \rangle = \bar{J}(t,x,u)$,  
where $u$ is defined as in \cref{lem:continuous_state_sec2}. 
This completes the proof.
\end{proof}

Note that if the parameters $\eta$, $\lambda$, $\phi$ and $\bar{\xi}$ are bounded, then uniqueness of the solution $(Y,Z)$ of \eqref{eq:BSDE_pen2} in the sense of \cref{def:sol_bsdes} holds by \cref{prop:sol_pen_problem}.
As an immediate consequence of \cref{lem:continuous_state_sec2} and \cref{prop:sol_pen_problem_eta_not_bounded_above}, we have that the solution constructed in \cref{prop:sol_pen_problem_eta_not_bounded_above} is the minimal solution under the conditions  \ref{assumption:A0}--\ref{assumption:A3} on the parameters. 

\begin{corollary}\label{cor:relax_upper_bound_minimal_sol}
     Assume that \ref{assumption:A0}--\ref{assumption:A3} are satisfied. 
     Then the solution $(Y,Z)$ constructed in \cref{prop:sol_pen_problem_eta_not_bounded_above} to the BSDE \eqref{eq:BSDE_pen2} is minimal, that is, if $(Y', Z')$ is another solution to the BSDE \eqref{eq:BSDE_pen2} in the sense of \cref{def:sol_bsdes}, then it holds P-a.s.\ for all $t \in [0,T]$ that $Y'_t \ge Y_t$.
\end{corollary}

In the next proposition we establish a lower bound for any solution of the BSDE~\eqref{eq:BSDE_pen2} under suitable assumptions on the coefficients and the terminal condition.
In particular, this gives a lower bound for the solution of \eqref{eq:BSDE_pen2} constructed in \cref{prop:sol_pen_problem_eta_not_bounded_above}.

\begin{propo}\label{prop:lower_bound_without_phi}
    Assume that \ref{assumption:A0}--\ref{assumption:A3} are satisfied, $\bar{\xi}$ is invertible $P$-a.s.,  
    and there exists $\delta_1 \in [0, \infty)$ such that $\| \eta^{-1}_t \phi_t + A_t \| \le \delta_1$ holds for all $t \in [0,T]$ $P$-a.s. Let $(Y,Z)$ be a solution to the BSDE \eqref{eq:BSDE_pen2} in the sense of \cref{def:sol_bsdes}.  
    Then for all $t \in [0,T]$ it holds $P$-a.s.\ that 
    $$ Y_t \ge e^{-4 \delta_1 T} E\bigg[ \bigg( \bar{\xi}^{-1} + \int_t^T \eta^{-1}_s ds \bigg)^{-1} \, \bigg\vert \, \cF_t \bigg].$$
\end{propo}

\begin{proof}
    We can rewrite the BSDE \eqref{eq:BSDE_pen2} as
    \begin{equation*}
       dY_t = \big( Y_t \eta_t^{-1} Y_t - \hat{\lambda}_t - \hat{A}^\top_t Y_t - Y_t\hat{A}_t \big) dt + Z_t d W_t, \quad 
       t \in [0,T], \quad 
       Y_T = \bar{\xi},
    \end{equation*}
    where $\hat{\lambda}_t = \lambda_t - \phi^\top_t \eta^{-1}_t \phi_t $, $\hat{A}_t = A_t + \eta^{-1}_t \phi_t$, $t \in [0,T]$. 
    It follows from \ref{assumption:A0} that $\hat{\lambda}=(\hat{\lambda}_t)_{t \in [0,T]}$ is $\mathbb{S}^d_+$-valued. Moreover, by assumption we have that $\| \hat{A}_t \| \le \delta_1$ for all $t \in [0,T]$.

    We start similarly to the proof of the upper bound in \cref{prop:sol_pen_problem}. 
    Let $U=(U_t)_{t\in[0,T]}$ be given by 
    $dU_t = -U_t \hat{A}_t dt$, $t\in[0,T]$, $U_0 = \Id$, and note that for all $t \in [0,T]$ it holds 
    \begin{equation}\label{eq:boundseigvalU}
        \sqrt{\lambda_{\min}(U^\top_t U_t)} \ge e^{- \delta_1 T} , \quad \sqrt{\lambda_{\min}((U^{-1}_t)^\top U^{-1}_t)} \ge e^{- \delta_1 T}, \quad \| U_t \| \le e^{\delta_1 T}, \quad \| U^{-1}_t \| \le e^{\delta_1 T}  .
    \end{equation}
    Define $M=(M_t)_{t\in[0,T]}$ by $M_t = (U_t^{-1})^\top Y_t U_t^{-1}$, $t \in [0,T]$. 
    Then we obtain that $M$ satisfies 
    $$ dM_t = M_t \left( \widetilde \eta_t \right)^{-1} M_t dt - \widetilde \lambda_t dt + (U^{-1}_t)^\top Z_t U_t^{-1} d W_t, \quad t \in [0,T], \quad M_T =\widetilde \xi,$$
    where for all $t \in [0,T]$ we have introduced 
    $$\widetilde \eta_t = (U^{-1}_t)^\top \eta_t U^{-1}_t,\qquad \widetilde \lambda_t = (U^{-1}_t)^\top \hat \lambda_t U_t^{-1}, \qquad \widetilde \xi = (U_T^{-1})^\top \bar{\xi} U_T^{-1}.$$
    Note that $\widetilde \eta =(\widetilde \eta_t)_{t\in[0,T]}$, $\widetilde \lambda = (\widetilde \lambda_t)_{t \in [0,T]}$, and $\widetilde \xi$ have the same integrability properties as $\eta$, $\lambda$, and $\bar{\xi}$ since $U^{-1}$ and $\eta^{-1} \phi$ are bounded. Moreover, $\widetilde{\eta}$ is bounded away from zero. 

    To obtain for all $t\in [0,T]$ a lower bound for $M_t$, we want to employ a comparison argument for matrix Riccati BSDEs (without linear terms) with bounded coefficients (see, e.g., \cite[Theorem~A.1]{horst2018multidim}). 
    We thus first perform a truncation using the same approach and notation as in the proof of \cref{prop:sol_pen_problem_eta_not_bounded_above} (replacing $A$ by $\hat{A}=(\hat{A}_t)_{t \in[0,T]}$, $\lambda$ by $\widetilde{\lambda}$, $\eta$ by $\widetilde{\eta}$, $\phi$ by zero, and $\bar{\xi}$ by $\widetilde\xi$). 
    Recall that for all $L\in\N$ we denote by $(M^L,Z^L)$ the unique solution of the Riccati BSDE with the bounded coefficients $\widetilde \eta^L$ and $\widetilde \lambda^L$ and the bounded terminal condition $\widetilde \xi^L$. 
    It follows from the proof of \cref{prop:sol_pen_problem_eta_not_bounded_above} and from \cref{cor:relax_upper_bound_minimal_sol} that 
    $M_t\ge M^L_t$ for all $t \in [0,T]$.

    We next define for all $L\in\N$ the processes $R^L=(R^L_t)_{t \in [0, T]}$ and $V^L=(V_t^L)_{t \in [0, T]}$ taking values in $\mathbb{S}^d_+$ as 
    $$R^L_t=\left( \big(\widetilde{\xi}^L\big)^{-1} + \int_t^T \big(\widetilde\eta_s^L\big)^{-1} ds \right)^{-1}, 
    \quad V^L_t = E\big[ R^L_t \, \big|\, \cF_t \big], \qquad t \in [0,T].$$
    Clearly, for all $L\in\N$ the process $R^L$ is a (non-adapted) solution to the Riccati equation 
    $$R^L_t = \widetilde{\xi}^L - \int_t^T R^L_s \big(\widetilde{\eta}_s^L\big)^{-1} R^L_s ds, \quad t \in [0,T].$$  
    This and Fubini's theorem yield for all $L\in\N$ and $t \in [0, T]$ that 
    \begin{align*}
            V^L_t &= E\big[ R^L_t \big| \cF_t \big] = 
            E\bigg[ \widetilde{\xi}^L - \int_t^T E\big[R^L_s \big(\widetilde\eta^L_s\big)^{-1} R^L_s \,\big|\, \cF_s \big] ds \,\bigg|\, \cF_t \bigg] \\
            & = E\bigg[ \widetilde{\xi}^L - \int_t^T V^L_s \big(\widetilde\eta^L_s)^{-1} V^L_s ds + \int_t^T \Gamma^L_s ds \,\bigg|\, \cF_t \bigg],
    \end{align*}
    with $\Gamma^L_s= E[ R^L_s \,|\, \cF_s] (\widetilde{\eta}^L_s)^{-1} E[ R^L_s \,|\, \cF_s] - E[R^L_s (\widetilde\eta_s^L)^{-1} R^L_s \,|\, \cF_s]$, $s\in [0,T]$.
    Thus, using the martingale representation theorem, we have that for all $L\in\N$ there exists an $\mathbb{S}^d$-valued, progressively measurable, square-integrable process $Z^{V,L}$ such that $(V^L, Z^{V,L})$ satisfies the BSDE 
    \begin{equation*}
        d V^L_t = \big(V^L_t \big(\widetilde\eta_t^L\big)^{-1} V^L_t - \Gamma^L_t \big) dt + Z^{V,L}_t dW_t, 
        \quad t\in[0,T], \quad V^L_T = \widetilde{\xi}^L. 
    \end{equation*}
    Note that $\Gamma^L_t$ is a.s.\ negative semidefinite for all $t \in [0, T]$ and $L\in\N$. Indeed, using Jensen's inequality, for any $L\in\N$, $x \in \R^d$, and $s \in [0,T]$, we obtain
    \begin{equation*}
    \begin{split}
        & \Big\langle x, E\big[R^L_s \,\big|\, \cF_s\big] \big(\widetilde\eta_s^L\big)^{-1} 
        E\big[ R^L_s \,\big|\, \cF_s\big] x \Big\rangle
        = 
        \left\|E\big[ \big(\widetilde\eta_s^L\big)^{-\frac{1}{2}} R^L_s x \,\big|\, \cF_s\big]\right\|^2
         \\
        & \le E\Big[ \big\lVert \big(\widetilde\eta_s^L\big)^{-\frac{1}{2}} R^L_s x \big\rVert^2 \,\Big|\, \cF_s \Big]
        = \Big\langle x, E\big[R^L_s \big(\widetilde\eta_s^L\big)^{-1} R^L_s \,\big|\, \cF_s \big] x \Big\rangle.
     \end{split}
    \end{equation*}
    Since $\Gamma^L_t$ is a.s.\ negative semidefinite and $\widetilde\lambda^L_t$ is a.s.\ positive semidefinite for all $t \in [0, T]$ and $L\in\N$, a BSDE comparison argument shows that $M^L_t \ge V_t^L$ for all $t \in [0, T]$ and $L\in\N$. 
    It follows for all $t \in [0, T]$ and $L\in\N$ that $M_t\ge E[ R^L_t \, | \, \cF_t]$. 
    By monotone convergence, we have for all $t \in [0,T]$ that the non-decreasing sequence $(E[ R^L_t \, | \, \cF_t])_{L\in\N}$ converges to 
    \begin{equation*}
        E\bigg[ \bigg( \widetilde{\xi}^{-1} + \int_t^T \widetilde\eta_s^{-1} ds \bigg)^{-1} \,\bigg|\, \cF_t \bigg] .
    \end{equation*}
    We therefore obtain for all $t\in[0,T]$, using \eqref{eq:boundseigvalU}, that 
    $$ M_t \ge  E\bigg[ \bigg( \widetilde{\xi}^{-1} + \int_t^T \widetilde\eta_s^{-1} ds \bigg)^{-1} \,\bigg|\, \cF_t \bigg] 
    \ge e^{-2 \delta_1 T} E\bigg[ \bigg( \bar{\xi}^{-1} + \int_t^T \eta_s^{-1} ds \bigg)^{-1} \,\bigg|\, \cF_t \bigg].$$ 
    The lower bound for $Y$ follows from this, \eqref{eq:boundseigvalU}, and $Y= U^\top M U$.
\end{proof}

In the next result, we establish an upper bound for the solution of \eqref{eq:BSDE_pen2} obtained in \cref{prop:sol_pen_problem_eta_not_bounded_above}. In general, this bound explodes at $T$ but is independent of $\bar{\xi}$. 
This bound becomes important in \cref{sec:sing_case}.

\begin{lemma}\label{lem:uniform_upper_bound}
    Assume that \ref{assumption:A0}--\ref{assumption:A3} are satisfied.
    Let $(Y,Z)$ be the solution of~\eqref{eq:BSDE_pen2} constructed in \cref{prop:sol_pen_problem_eta_not_bounded_above}. 
    Then for all $t\in[0,T)$ 
    it holds  $P$-a.s.\ that 
    \begin{equation} \label{eq:a_priori_estim}
     \begin{split}
        Y_t   
        &\leq  \frac{1}{(T-t)^2} E\bigg[ \int_t^T \big( (\Id - A_s (T-s))^\top \eta_s (\Id - A_s (T-s)) + \lambda_s (T-s)^2 \\
        & \qquad \qquad \qquad \qquad + 2 \phi^\top_s  (T -s)(\Id - A_s (T-s)) \big) \, ds  \,\bigg|\, \cF_t \bigg]   . 
     \end{split}   
    \end{equation}
\end{lemma}

\begin{proof}
    Throughout the proof fix $t\in [0,T)$ and $x\in \R^d$. 
    By \cref{prop:sol_pen_problem_eta_not_bounded_above} 
    we have $\bar{v}(t,x)=\langle x,Y_t x\rangle$. 
    Let $X_s=\frac{T-s}{T-t} x$, $s\in [t,T]$, and
    define the control $u\in \mathcal A(t,x)$ by $u_s=\frac{1}{T-t}x + A_s X_s$, $s\in [t,T]$. Then it holds that $dX_s=(A_s X_s - u_s)ds$, $X_T=0$, $\langle x,Y_t x\rangle \le \bar J(t,x,u)$, and 
    \begin{equation*}
     \begin{split}
        \bar J(t,x,u) &= \frac{1}{(T-t)^2}\bigg\langle x, E\bigg[ \int_t^T \big( \left(\Id + A_s (T-s)\right)^\top \eta_s \left(\Id + A_s (T-s)\right) + \lambda_s (T-s)^2 \\
        & \qquad \qquad \qquad \qquad \qquad+ 2 \phi^\top_s  (T -s)(\Id + A_s (T-s)) \big) ds  \,\Big|\, \cF_t \bigg] x \bigg\rangle  . 
     \end{split}   
    \end{equation*}
    This completes the proof.
\end{proof}

\section{The singular case}\label{sec:sing_case}

We now return to the setting of \cref{sec:prob_form} and introduce 
an additional assumption that, together with \ref{assumption:A0}--\ref{assumption:A2}, is 
sufficient for the construction of a supersolution to the BSDE~\eqref{eq:BSDE_sing}.
\begin{enumerate}[label=\textbf{B\arabic*)}]
    \item\label{assumption:B2}
    The process $\eta$ has a moment of order three, that is, 
        $$E\bigg[ 
        \int_0^T \|\eta_r\|^3 \, dr \bigg] < \infty.$$
\end{enumerate}
We recall that $C=\ker(\xi)$ for some $\mathcal F_T$-measurable, $\mathbb{S}_+^d$-valued, bounded random matrix~$\xi$. 
For the matrix $\theta$, we again use the truncation technique of \cref{lem:measurable_truncation} 
to define a non-decreasing sequence $(\theta^n)_{n\in\N}$ of bounded matrices converging to $\theta$. 
In the sequel we use the results of \cref{sec:pen_case} in the setting of \cref{sec:prob_form} with the choice $\bar \xi=n\xi+\theta^n$, $n\in \N$. In particular, we denote by $(Y^n,Z^n)$ the solution of \eqref{eq:BSDE_pen} constructed in \cref{prop:sol_pen_problem_eta_not_bounded_above}.

\begin{propo}\label{prop:sol_convergence}
Assume that \ref{assumption:A0}--\ref{assumption:A2} and~\ref{assumption:B2} are satisfied.
For every $n\in \N$ let $(Y^n,Z^n)$ be the solution of~\eqref{eq:BSDE_pen} in the sense of \cref{def:sol_bsdes} constructed in \cref{prop:sol_pen_problem_eta_not_bounded_above}. 
Then there exists a supersolution $(Y, Z)$ to the BSDE~\eqref{eq:BSDE_sing} in the sense of \cref{def:supersol_bsde} such that for every $t \in [0, T)$ the sequence of random variables $(Y^n_t)_{n\in\N}$ converges $P$-a.s.\ to $Y_t$ and $(Z^n)_{n\in\N}$ converges in $L^2([0, t]\times \Omega)$ to $Z$.
Moreover, the estimate \eqref{eq:a_priori_estim} holds for $Y$.
\end{propo}

\begin{proof} 
From the value function representation in \cref{prop:sol_pen_problem_eta_not_bounded_above} we conclude 
for every $t\in [0,T]$ that the sequence $(Y^n_t)_{n\in \N}$ is non-decreasing. 
Moreover, \cref{lem:uniform_upper_bound} and our integrability assumptions \ref{assumption:A1}, \ref{assumption:A2} and \ref{assumption:B2} on $A$, $\eta$, $\phi$ and $\lambda$ ensure for all $t\in[0,T)$ and $P$-almost all $\omega \in \Omega$ that there exists $c_t(\omega) \in (0,\infty)$ such that for all $n\in\N$ it holds $Y_t^n(\omega)\le c_t(\omega)$. 
Hence, by \cref{lemma:matrix_convergence}, for any $t\in[0,T)$ the sequence $(Y_t^n)_{n\in\N}$ converges $P$-a.s.\ to $Y_t$. 
Note that $Y=(Y_t)_{t \in [0,T)}$ still satisfies the estimate~\eqref{eq:a_priori_estim} and is $\mathbb{S}_+^d$-valued.
In particular, it follows from this, \ref{assumption:A1}, \ref{assumption:A2}, \ref{assumption:B2}, and Jensen's inequality for all $t \in [0, T)$ that $Y_t \in L^3(\Omega)$ and $Y \in L^3(\Omega \times [0,t])$.

Assume that $(\psi_t)_{t\in [0, T]}$ is a process that is left-continuous at time $T$. Note that $P$-a.s.\ for all $n\in \N$ it holds that 
\begin{equation}\label{eq:sing_term_condition}
    \liminf_{t \to T} \langle \psi_t, Y_t \psi_t \rangle \geq \liminf_{t\to T} \langle \psi_t,  Y^n_t \psi_t \rangle = n \langle \psi_T, \xi \psi_T \rangle + \langle \psi_T, \theta^n \psi_T \rangle.
\end{equation}
If $\psi_T \notin \ker(\xi)$, sending $n \to \infty$ we obtain 
$ \liminf_{t \to T} \langle \psi_t, Y_t \psi_t \rangle = +\infty.$ 
Furthermore, \eqref{eq:sing_term_condition} implies that if $\psi_T \in \ker(\xi)$, we have 
$\liminf_{t \to T} \langle \psi_t, Y_t \psi_t \rangle \ge \langle \psi_T, \theta \psi_T \rangle$. 
Thus, $Y$ satisfies the singular terminal condition in the sense of \cref{def:supersol_bsde}. 

To establish that there exists an appropriate process $Z$ such that $(Y,Z)$ satisfies the BSDE~\eqref{eq:BSDE_sing}, we proceed similar to the proof of \cref{prop:sol_pen_problem_eta_not_bounded_above}. 
For all $n, m \in \N$ and $s, t \in [0, T]$ with $t > s$, It\^o's formula leads to 
\begin{align}\label{eq:Ito_form}
    &\lVert Y^n_s - Y^m_s \rVert^2 + \int_s^t 
    \lVert Z^n_r - Z^m_r \rVert^2 dr \\ \nonumber
    &=\lVert Y^n_t - Y^m_t\rVert^2 - 
    2\int_s^t \tr [(Y^n_r - Y^m_r)(Z^n_r - Z^m_r)]dW_r \\ \nonumber
    & \quad + 2 \int_s^t \tr \big[(Y^{ n}_r - Y^{ m}_r)(Y^n_r A_r + A_r^\top Y^n_r - Y^m_r A_r - A_r^\top Y^m_r) \big] dr \\
    & \quad  + 2\int_s^t \tr \big[ (Y^n_r - Y^m_r)(-(Y^n_r - \phi_r^\top)\eta^{-1}_r(Y^n_r - \phi_r) + (Y^m_r - \phi_r^\top)\eta^{-1}_r(Y^m_r - \phi_r))\big]dr. \nonumber 
\end{align}
Since for all $n\in\N$ we have $Y^n \in \mathcal{L}^2(\Omega, C([0, T], \mathbb{S}^d))$ 
and $Z^n \in L^2([0,T]\times \Omega)$, 
we obtain for all $n,m\in\N$ and $t\in[0,T]$ that 
\begin{equation}\label{eq:Ito_form_mart_comp}
    E\bigg[\int_0^t \tr\big[ (Y^n_r - Y^m_r)(Z^n_r - Z^m_r)\big]dW_r\bigg] = 0.
\end{equation}
Recall that for all $r \in [0,T]$ it holds $\lVert \eta_r^{-1} \rVert \le \sqrt{d} \delta^{-1}$. 
It thus follows from Jensen's inequality for all $n,m \in \N$ and $r \in [0,T)$ that 
\begin{equation*}
    \begin{split}
        & \tr \big[ (Y^n_r - Y^m_r)(-(Y^n_r - \phi_r^\top)\eta^{-1}_r(Y^n_r - \phi_r) + (Y^m_r - \phi_r^\top)\eta^{-1}_r(Y^m_r - \phi_r))\big] \\
        & \le \lVert Y^n_r - Y^m_r \rVert 
        (\lVert Y^n_r -\phi_r \rVert^2 \lVert \eta_r^{-1}\rVert + \lVert Y^m_r -\phi_r \rVert^2 \lVert \eta_r^{-1}\rVert) \\
        & \le 8 \sqrt{d}\delta^{-1} \lVert Y_r \rVert \big( \lVert Y_r \rVert^2 + \lVert \phi_r \rVert^2 \big) .
    \end{split}
\end{equation*}
Since $Y \in L^3(\Omega \times [0,t])$ for all $t \in [0, T)$ and $\phi \in L^3(\Omega \times [0,T])$ (cf.\ \ref{assumption:A2}), we can apply the dominated convergence theorem to obtain for all $t\in[0,T)$ that 
\begin{equation}\label{eq:DCT_lim_of_integral1}
    \begin{split}
        & \lim_{n,m\to\infty} E\bigg[ \int_0^t  
        \tr \big[ (Y^n_r - Y^m_r) \big(-(Y^n_r - \phi_r^\top)\eta^{-1}_r(Y^n_r - \phi_r) \\
        & \qquad \qquad \qquad \quad 
        + (Y^m_r - \phi_r^\top)\eta^{-1}_r(Y^m_r - \phi_r)\big)\big]
        \, dr \bigg] = 0 .
    \end{split}
\end{equation}
Moreover, due to \ref{assumption:A1} we have for all $n,m \in \N$ and $r \in [0,T)$ that 
\begin{equation*}
    \begin{split}
        & \tr \big[(Y^{ n}_r - Y^{ m}_r)(Y^n_r A_r + A_r^\top Y^n_r - Y^m_r A_r - A_r^\top Y^m_r) \big]
        \le 4 K \lVert Y^{ n}_r - Y^{ m}_r \rVert \lVert Y_r \rVert 
        \le 8  K \lVert Y_r \rVert^2.
    \end{split}
\end{equation*}
This, the fact that $Y \in L^3(\Omega \times [0,t])$ for all $t \in [0, T)$, and the dominated convergence theorem show for all $t \in [0,T)$ that 
\begin{equation}\label{eq:DCT_lim_of_integral2}
    \begin{split}
        \lim_{n,m\to\infty} E\bigg[ \int_0^t \tr \big[(Y^{ n}_r - Y^{ m}_r)(Y^n_r A_r + A_r^\top Y^n_r - Y^m_r A_r - A_r^\top Y^m_r) \big]\, dr \bigg] = 0 .
    \end{split}
\end{equation}
In addition, since $Y_t \in L^3(\Omega)$ for all $t \in [0, T)$, the dominated convergence theorem implies for all $t\in[0,T)$ that $\lim_{n,m\to\infty} E[\lVert Y_t^n - Y_t^m \rVert^2 ] = 0$. 
Combining this and \eqref{eq:Ito_form}--\eqref{eq:DCT_lim_of_integral2} yields for all $t\in[0,T)$ that
\begin{equation*}
    \lim_{n,m \to \infty}E\bigg[\int_0^t \lVert Z^n_r - Z^m_r \rVert^2 dr\bigg] = 0.
\end{equation*}
Hence, there exists a progressively measurable process $Z\colon [0,T)\times \Omega \to \mathbb{S}^d$ such that for all $t \in [0,T)$ the sequence $(Z^n)_{n\in\N}$ converges to $Z$ in $L^2([0, t]\times \Omega)$. 
Now, arguing as in the proof of \cref{prop:sol_pen_problem_eta_not_bounded_above}, we obtain for all $t\in[0,T)$ that 
$$\lim_{n,m\to \infty}E\Big[ \sup_{s \in [0,t]} \lVert Y^n_s - Y^m_s \rVert^2\Big] = 0.$$ 
We can further show that $(Y,Z)$ satisfies \eqref{eq:supersoln_BSDE_equation}. 
This, together with \eqref{eq:sing_term_condition}, establishes that the constructed process $(Y,Z)$ solves the BSDE~\eqref{eq:BSDE_sing} in the sense of \cref{def:supersol_bsde}. 
\end{proof}

Next, we present a lemma that establishes properties of solutions of~\eqref{eq:BSDE_sing} in the sense of \cref{def:supersol_bsde} and the connection of these solutions to the control problem \eqref{eq:objective} with the terminal state constraint \eqref{eq:constraint}.

\begin{lemma}\label{lemma:continuous_state}
    Assume that \ref{assumption:A0}--\ref{assumption:A2} are satisfied
    and suppose that there exists a supersolution 
    $(Y, Z)$ to the BSDE \eqref{eq:BSDE_sing} in the sense of \cref{def:supersol_bsde}.
    Let $x\in\R^d$ and $t\in[0,T)$.  
    Let $u=(u_s)_{s\in[t,T]}$ satisfy $u_s=\eta_s^{-1} (Y_s - \phi_s) X_s$, $s\in[t,T)$, where $(X_s)_{s\in[t,T)}$ is the unique solution of the differential equation with random coefficients 
    \begin{equation}\label{eq:diffequwithrandomcoeffsing}
        dX_s= (- \eta_s^{-1} (Y_s - \phi_s) + A_s) X_s ds, \quad s\in[t,T), \quad X_t=x.
    \end{equation}
    Then: 

    (i) It holds that $E[\int_t^T \lVert u_s \rVert^2 ds]<\infty$. 

    (ii) For $X^u=(X^u_s)_{s\in[t,T]}$ defined by \eqref{eq:dynamics} it holds $P$-a.s.\ that $X_T^u \in \ker(\xi)$.

    (iii) It holds $P$-a.s.\ that $\langle x, Y_t x \rangle \ge J(t,x,u)$.
\end{lemma}

\begin{proof}
First note that since $A$ and $\eta^{-1}$ are bounded and for all $r\in[t,T)$ it holds that  $E[\int_t^r \lVert Y_s - \phi_s \rVert ds]<\infty$,
there indeed exists a unique solution $(X_s)_{s\in[t,T)}$ of~\eqref{eq:diffequwithrandomcoeffsing}.
 
Next, observe that for all $r \in [t,T)$ and all $[t,r]$-valued stopping times $\tau$, It\^o's formula shows that 
    \begin{equation}\label{eq:ito_form_supermart}
        \begin{split}
            \langle X_{\tau}, Y_{\tau} X_{\tau} \rangle 
            & = \langle x, Y_t x\rangle 
            - \int_{t}^{\tau} \big(\langle X_s, \lambda_s X_s \rangle  + \langle u_s, \eta_s u_s \rangle ds + 2 \langle \phi_s X_s, u_s \rangle\big) ds \\
            & \quad + \int_{t}^{\tau} \langle X_s,Z_sX_s \rangle dW_s .
        \end{split}
    \end{equation}
We introduce a non-increasing sequence  $(\varepsilon_m)_{m \in \N}$ such that $\varepsilon_m \in (0, T -t)$, $m \in \N$, and $\lim_{m \to \infty}\varepsilon_m = 0$. Set $T_m: = T - \epsilon_m$, $m\in\N$. 
Note that for each $m\in\N$, the process $(\int_t^{r} \langle X_s,Z_sX_s \rangle dW_s)_{r\in[t,T_m]}$ is a local martingale.  
For each $m\in\N$ let $(\tau_{k,m})_{k\in\N}$ be a localizing sequence of stopping times converging monotonically to $T_m$. 
In particular, we have for all $m,k\in\N$ that  
$$ E\bigg[ \int_t^{\tau_{k,m}} \langle X_s, Z_s X_s \rangle d W_s \,\bigg|\, \cF_t \bigg] = 0 .$$
It thus follows from \eqref{eq:ito_form_supermart} for all $m,k\in\N$ that 
\begin{equation*}
    \begin{split}
        \langle x, Y_t x \rangle &= E \bigg[\int_t^{\tau_{k,m}} \! \big(\langle u_r, \eta_r u_r \rangle + 2\langle \phi_r X_r, u_r \rangle + \langle X_r, \lambda_r X_r \rangle\big) dr+ \langle X_{\tau_{k,m}}, Y_{\tau_{k,m}} X_{\tau_{k,m}} \rangle \bigg| \cF_t \bigg] .
    \end{split}
\end{equation*}
For any $m \in \N$, taking the limit $k \to \infty$, we obtain by combining \ref{assumption:A0}, the monotone convergence theorem, and Fatou's lemma that 
    \begin{align} \nonumber
        \langle x, Y_t x \rangle 
        & \ge 
        E \bigg[\int_t^{T_m} \! \big(\langle u_r, \eta_r u_r \rangle + 2\langle \phi_r X_r, u_r \rangle + \langle X_r, \lambda_r X_r \rangle\big) dr 
        + \langle X_{T_m}, Y_{T_m} X_{T_m} \rangle
        \,\bigg|\, \cF_t \bigg] \\
        & \ge E \bigg[\int_t^{T_m} \! \big(\langle u_r, \eta_r u_r \rangle + 2\langle \phi_r X_r, u_r \rangle + \langle X_r, \lambda_r X_r \rangle\big) dr 
        \,\bigg|\, \cF_t \bigg] .
    \label{eq:costs_localization}
    \end{align}
Taking the limit $m \to \infty$ and applying the monotone convergence theorem one more time, we conclude that 
\begin{equation*}
        \langle x, Y_t x \rangle \ge   E \bigg[\int_t^{T} \big(\langle u_r, \eta_r u_r \rangle + 2\langle \phi_r X_r, u_r \rangle + \langle X_r, \lambda_r X_r \rangle\big) dr \,\bigg|\, \cF_t \bigg].
\end{equation*}    
This and \ref{assumption:A0} yield
\begin{equation}\label{eq:control_in_L2}
    \delta E\bigg[\int_t^T \lVert u_r \rVert^2 dr\bigg] 
    \le E\big[ \langle x, Y_t x\rangle \big] 
    \le \lVert x \rVert^2 E[\textstyle{\sup_{0 \le s \le t}} \lVert Y_s \rVert]
    < \infty.
\end{equation} 
Let $X^u=(X^u_s)_{s\in[t,T]}$ be defined by \eqref{eq:dynamics} and note that it holds $P$-a.s.\ for all $s\in[t,T)$ that $X^u_s=X_s$ and $\lim_{r\to T} X_r = X_T^u$. 
To show that $X_T^u \in \ker(\xi)$, observe that \eqref{eq:ito_form_supermart} together with \ref{assumption:A0} and $E[ \lvert \langle x, Y_t x\rangle \rvert ] <\infty$ 
establishes that $\langle X^u, Y X^u \rangle$ is a local supermartingale on $[t,T)$.
Since $Y$ is $\mathbb{S}_+^d$-valued, it holds that $\langle X^u, Y X^u \rangle$ is a nonnegative supermartingale on $[t,T)$. 
It follows from the supermartingale convergence theorem that 
there exists a random variable $\overline{X} \in L^1(\Omega,\R)$ such that $\lim_{s \to T} \langle X^u_s, Y_s X^u_s \rangle = \overline{X}$ $P$-a.s.
This, continuity of the state process $X^u$, and the fact that $(Y,Z)$ is a supersolution of~\eqref{eq:BSDE_sing} in the sense of \cref{def:supersol_bsde} imply that $X^u_T \in \ker(\xi)$ almost surely. 
From $X^u_T \in \ker(\xi)$ $P$-a.s.\ and~\eqref{eq:control_in_L2} we conclude that $(u_s)_{s \in [t,T]}\in \mathcal A_0(t,x)$. 
Furthermore, note that by \cref{def:supersol_bsde} and continuity of $X^u$ we have $\lim_{s \to T} \langle X^u_s, Y_s X^u_s \rangle \ge \langle X^u_T, \theta X^u_T \rangle$. 
Fatou's lemma, the monotone convergence theorem, and \eqref{eq:costs_localization} therefore imply that 
\begin{equation*}
    \begin{split}
        & \langle x, Y_t x \rangle \\
        &\ge E \bigg[\int_t^{T} \! \big(\langle u_r, \eta_r u_r \rangle + 2 \langle \phi_r X^u_r, u_r \rangle + \langle X^u_r, \lambda_r X^u_r \rangle \big) dr + \liminf_{m \to \infty}\langle X^u_{ T_m}, Y_{ T_m} X^u_{T_m} \rangle \,\bigg|\, \cF_t \bigg]\\
        & \ge E \bigg[\int_t^{T} \! \big(\langle u_r, \eta_r u_r \rangle +  2 \langle \phi_r X^u_r, u_r \rangle + \langle X^u_r, \lambda_r X^u_r \rangle\big)  dr 
        + \langle X_{T}^u, \theta X^u_{T} \rangle \,\bigg|\, \cF_t \bigg] 
        = J(t,x,u).
    \end{split}
    \end{equation*}
This completes the proof.
\end{proof}

In the following result we provide 
the desired representation of the value function~$v$ and a characterization of optimal strategies in terms of the supersolution~$Y$ constructed in \cref{prop:sol_convergence}.

\begin{propo}\label{prop:sol_sing_problem}
Assume that \ref{assumption:A0}--\ref{assumption:A2} and \ref{assumption:B2} are satisfied. Let $(Y,Z)$ be the supersolution to the singular BSDE~\eqref{eq:BSDE_sing} constructed in \cref{prop:sol_convergence}. 
Then we have 
for all $x \in \R^d$ and $t \in [0, T)$ $P$-a.s.\  that $v(t,x)=\langle x,Y_t x\rangle$, and the optimal state process $X^*$ satisfies $\dot{X}^{*}_s=-\eta_s^{-1}(Y_s - \phi_s) X^{*}_s + A_s X^*_s$, $s\in [t,T)$, $X_t = x$.
\end{propo}

\begin{proof}
For all $n \in \N$ let $(Y^n,Z^n)$ be the solution to the BSDE \eqref{eq:BSDE_pen} constructed in \cref{prop:sol_pen_problem_eta_not_bounded_above}. Note that for all $n\in \N$, $t\in[0,T)$, and $x\in\R^d$ it holds $P$-a.s.\ that 
\begin{equation*} 
v(t,x)=\essinf_{u\in \mathcal A_0(t,x)}J(t,x,u)
\ge \essinf_{u\in \mathcal A_0(t,x)}J_n(t,x,u)
\ge \essinf_{u\in \mathcal A(t,x)}J_n(t,x,u)
=\langle x, Y^n_t x\rangle.
\end{equation*}
This together with \cref{prop:sol_convergence} implies for all $t\in[0,T)$ and $x\in\R^d$ $P$-a.s.\ that $v(t,x)\ge \langle x, Y_t x\rangle$. Combining this with \cref{lemma:continuous_state}, we obtain, for all $x \in \R^d$ and $t \in [0, T)$, that $P$-a.s.\ $v(t,x)=\langle x,Y_t x\rangle=J(t,x,u)$, where $u$ is defined as in \cref{lemma:continuous_state}. 
Uniqueness of optimal control and optimal state follow from the uniform convexity of $J$ (see also \cref{rem:uniform_convexity}) and \cref{lemma:continuous_state}, respectively.
\end{proof}

The next result establishes minimality of the supersolution constructed in \cref{prop:sol_convergence}. This result is a straightforward corollary of \cref{lemma:continuous_state} and \cref{prop:sol_sing_problem}.

\begin{corollary}\label{cor:minimal_solution_new}
Assume that \ref{assumption:A0}--\ref{assumption:A2} and \ref{assumption:B2} are satisfied. 
Let $(Y, Z)$ be the supersolution to the singular BSDE \eqref{eq:BSDE_sing} in the sense of \cref{def:supersol_bsde} constructed in \cref{prop:sol_convergence}. Assume that $(Y', Z')$ is another supersolution to \eqref{eq:BSDE_sing} in the sense of \cref{def:supersol_bsde}. Then it holds $P$-a.s.\ for all $t \in [0, T)$ that $Y_t' \geq  Y_t$.
\end{corollary}

Recall that the random variable $\xi$ satisfying $\ker(\xi)=C$ is not unique.
Since the construction of the supersolution $(Y,Z)$ to the singular BSDE \eqref{eq:BSDE_sing} in \cref{prop:sol_convergence} relies on the specific choice of such a $\xi$, it is a priori conceivable that $(Y,Z)$ might depend on this choice.
As an immediate consequence of \cref{cor:minimal_solution_new} we obtain that this is not the case.

\begin{corollary}\label{cor:independence_of_xi}
Assume that \ref{assumption:A0}--\ref{assumption:A2} and \ref{assumption:B2} are satisfied. 
Let $\xi_1,\xi_2\colon \Omega \to \mathbb{S}_+^d$ be $\cF_T$-measurable and bounded and suppose that $\ker(\xi_1) = \ker(\xi_2)$. 
Let $(Y^1,Z^1)$ and $(Y^2,Z^2)$ be the supersolutions to the singular BSDE~\eqref{eq:BSDE_sing} in the sense of \cref{def:supersol_bsde} constructed in \cref{prop:sol_convergence} with $\xi_1$ and $\xi_2$, respectively, as the approximating terminal conditions. Then $Y^1=Y^2$.
\end{corollary}

\begin{remark}\label{rem:howtoincludeB}
    Let $B \colon [0,T]\times \Omega \to \mathbb \R^{d \times d}$ be a progressively measurable process. Generalizing~\eqref{eq:dynamics} one could consider 
    state processes $$\dot{X}_s = B_s u_s + A_s X_s,\quad s \in [t, T],\quad X_t = x.$$
    This leads to the slightly modified BSDEs
\begin{equation*}
\begin{split}
    dY_t &= \bigl((Y_t B_t + \phi^{\top}_t) \eta_t^{-1} (B_t^{\top} Y_t + \phi_t) - \lambda_t - Y_t A_t - A^{\top}_t Y_t \bigr)dt+Z_tdW_t, 
    \quad t \in [0,T],\\
    Y_T &=\theta + \xi \quad (\text{or, in the singular case, } Y_T=\infty \Ind_{C^c}+\theta \Ind_C).
\end{split}
\end{equation*}
The candidate for the optimal state process satisfies
$$\dot{X}_s = -B_s \eta_s^{-1}(B^\top_s Y_s + \phi_s) X_s + A_s X_s,  \quad s \in [t, T),\quad X_t = x.$$ 
One can then 
show that the results for the penalized case hold under the additional assumption that there exists $\overline{K} \in (0, \infty)$ such that for all $s \in [0,T]$ it holds $\|B_s\| \le \overline{K}$. For the results of \cref{sec:sing_case} the upper bound constructed in \cref{lem:uniform_upper_bound} is crucial; hence some controllability conditions are necessary to proceed in the same way. 
The results in the singular case can be obtained 
under conditions \ref{assumption:A0}--\ref{assumption:A2} and 
\ref{assumption:B2} and the additional assumption that for all $s \in [0,T]$, the random matrix $B_s$ is $P$-a.s.\ invertible and there exists $\overline{K} \in (0, \infty)$ such that $(\|B_s\| + \| B^{-1}_s \|) \le \overline{K}$.
\end{remark}

\begin{remark}
Note that the results of this section also apply to the setting of \cref{sec:pen_case}. In this way, we can obtain variants of the results of \cref{sec:pen_case} without imposing Condition~\ref{assumption:A3}. Indeed, by setting $\xi = 0$ (so that $C=\R^d$) the control problem \eqref{eq:objective} features no terminal state constraint, but only terminal costs $\langle X_T,\theta X_T\rangle$. Hence, we are in the setting of \cref{sec:pen_case} with $\bar \xi=\theta$. 
The present section can now be used to provide a characterization of the optimal control $u^*$ and the corresponding state $X^*$ via the supersolution $Y$ of~\eqref{eq:BSDE_sing} with the terminal condition $\bar \xi=\theta$ solely under the condition that $\theta$ is $P$-a.s.\ finite (without requiring finite third moments as in Condition~\ref{assumption:A3}). Note, however, that the process $Y$ is only defined on $[0,T)$.
\end{remark}

\section{The case $C=\{0\}$} \label{sect:invertible_case}

In this section we assume that 
the kernel of~$\xi$ only contains the zero vector  
(i.e., $\xi$ is invertible) and present further results on the minimal solution~$Y$
of the singular BSDE~\eqref{eq:BSDE_sing} and the corresponding optimal control~$u^*$. Two cases are considered: 
\begin{itemize}
    \item When~$\eta$ has uncorrelated multiplicative increments (see \cref{def:uncorrmultincrem}) and all other coefficients vanish, we show that~$Y$ admits a closed-form expression and~$u^*$ is deterministic. We thus generalize the results of \cite[Section 5]{ankirchner2014bsdes} to the multidimensional case. Note that in this framework, for the existence of the minimal solution~$Y$, the integrability condition on~$\eta$ can be relaxed (Assumption~\ref{assumption:C1} below instead of~\ref{assumption:A0} and~\ref{assumption:B2}). 
    
    \item When all coefficients are bounded and~$\eta$ is an It\^o process, $Y$ admits an explicit asymptotic expansion in a neighborhood of the terminal time~$T$. Here we generalize \cite[Section 4.1]{grae:popi:21} to the multidimensional case. In dimension~1, this expansion has also been used to prove the existence of the solution of the HJB equation in \cite{graewe2015non,graewe2018smooth,horst2016constrained}.
\end{itemize}

Before proceeding, we state the following consequence of \cref{prop:lower_bound_without_phi}, which is used in \cref{ssect:eta_ito_process}.

\begin{lemma} \label{lem:lower_bound}
    Assume that \ref{assumption:A0}--\ref{assumption:A2} are satisfied, that $\xi$ is $P$-a.s.\ invertible, and that there exists $\delta_1 \in [0, \infty)$ such that $\| \eta^{-1}_t \phi_t + A_t \| \le \delta_1$ holds for all $t \in [0,T]$ $P$-a.s. 
    Let $(Y,Z)$ be the supersolution to the singular BSDE~\eqref{eq:BSDE_sing} constructed in \cref{prop:sol_convergence}. 
    Then for all $t\in[0,T)$ it holds $P$-a.s.\ that 
    $$Y_t \geq e^{-4 \delta_1 T} E \bigg[ \bigg( \int_t^T \eta^{-1}_s ds \bigg)^{-1} \,\bigg|\, \cF_t \bigg] .$$
\end{lemma}

\begin{proof}
For all $n\in\N$ take $\xi^n = n \xi$ and $\theta=0$, and let $(Y^n,Z^n)$ be the solution to the BSDE~\eqref{eq:BSDE_pen} constructed in \cref{prop:sol_pen_problem_eta_not_bounded_above} with the terminal condition $\xi^n$. 
Note that for all $t\in [0,T)$ it holds that $Y_t$ is the $P$-a.s.\ limit of the non-decreasing sequence $(Y_t^n)_{n\in\N}$. 
This and \cref{prop:lower_bound_without_phi} imply for all $n\in\N$ and $t\in[0,T)$ that 
\begin{equation*}
    Y_t \ge Y_t^n 
    \ge e^{-4 \delta_1 T} E \bigg[ \bigg( \frac{1}{n} \xi^{-1} + \int_t^T \eta^{-1}_s ds \bigg)^{-1} \,\bigg|\, \cF_t \bigg] .
\end{equation*}
Applying Fatou's lemma yields the claim.
\end{proof}

\subsection{Processes with uncorrelated multiplicative increments} \label{ssect:umi}

In this section we consider the singular BSDE~\eqref{eq:BSDE_sing} under the following setting:
\begin{enumerate}[label=\textbf{C\arabic*)}]
    \item \label{assumption:C0}
    The processes 
    $\lambda, \phi, A$ 
    are zero   
    and it holds 
    $\xi = \Id$.
    
    \item\label{assumption:C1}
    The process $\eta$ 
    satisfies for all 
    $0\le s \le t \le T$ that  $E[\|\eta_t \|] < \infty$, $E[\| \eta^{-1}_s \eta_t \|] < \infty$, and
    $$E\bigg[ \| \eta_T\|^2 + \int_0^T \big(\|\eta_r\|^2 + \|\eta^{-1}_r\| \big) dr \bigg] < \infty.$$
\end{enumerate}

\begin{defi}\label{def:uncorrmultincrem}
Under \ref{assumption:C1} we say that $\eta$ has uncorrelated multiplicative increments if for all 
$0\le s \le t \le T$
it holds 
\begin{equation}\label{eq:def_uncorr_mult_increm}
    E\big[\eta^{-1}_s \eta_t \,\big|\, \mathcal{F}_s\big] = E\big[\eta^{-1}_s \eta_t\big]. 
\end{equation}
\end{defi}
 
Note that in order to obtain in \cref{prop:sol_bsde_uncor_mult} below a solution to the BSDE, we impose integrability conditions similar to those assumed in the one-dimensional case (see \cite[Section 5]{ankirchner2014bsdes}). The main technical difference is that, in the present multidimensional setting, the state processes are not necessarily monotone nor bounded. 

To obtain in \cref{prop:mult_uncorr_val_fun} below 
the characterization of the value function, we assume in addition to \ref{assumption:C0} and \ref{assumption:C1} that the process $\eta$ is uniformly bounded away from zero: We impose ~\ref{assumption:A0}, which 
under \ref{assumption:C0} is equivalent to the assumption that there exists $\delta \in (0,\infty)$ such that $P$-a.s.\ for all $t \in [0,T]$ it holds $\eta_t \ge \delta \Id$.

We first characterize when $\eta$ has uncorrelated multiplicative increments.

\begin{lemma}\label{prop:process_mult_incr}
Assume that \ref{assumption:C1} is satisfied.
Then:

(i) If $\eta$ has uncorrelated multiplicative increments, then it holds for all $0\le s\le t \le T$ that 
$E[\eta_s^{-1} \eta_t] = (E[\eta_s])^{-1} E[\eta_t]$.

(ii) The process $\eta$ has uncorrelated multiplicative increments if and only if the process $(\eta_t (E[\eta_t])^{-1})_{t \in [0,T]}$ is a martingale. 
\end{lemma}

\begin{proof}
If $\eta$ has uncorrelated multiplicative increments, then we have for all $0\le s \le t \le T$ that 
\begin{equation}\label{eq:umi_aux}
    E[\eta_t] = E[\eta_s E[\eta_s^{-1} \eta_t  | \mathcal{F}_s] ] = E[\eta_s] E[\eta_s^{-1} \eta_t ] .
\end{equation}
This proves (i).
Let $M_t=\eta_t (E[\eta_t])^{-1}$, $t\in[0,T]$.
If $\eta$ has uncorrelated multiplicative increments, then it follows from \eqref{eq:def_uncorr_mult_increm} and \eqref{eq:umi_aux} for all $0\le s \le t \le T$ that 
\begin{equation*}
    \begin{split}
        E[M_t | \mathcal{F}_s ] & =  E[\eta_t | \mathcal{F}_s] (E[\eta_t])^{-1} = \eta_s E[ \eta_s^{-1} \eta_t   | \mathcal{F}_s] (E[\eta_t])^{-1}  =   \eta_s  (E[\eta_s])^{-1} = M_s;
    \end{split}
\end{equation*}
hence, $M$ is a martingale. 
For the converse direction, assume that 
$(M_t)_{t\in[0,T]}$ 
is a martingale. 
Then we have for all $0\le s \le t \le T$ that 
\begin{equation*}
    \begin{split}
        \eta_s E[\eta_s^{-1} \eta_t| \cF_s ] 
        & = E[M_t | \mathcal{F}_s] E[\eta_t] =  M_s E[\eta_t]=   \eta_s (E[\eta_s])^{-1} E[\eta_t].
    \end{split}
\end{equation*} 
Hence, it holds for all $0\le s \le t \le T$ that 
$E[ \eta^{-1}_s \eta_t  | \mathcal{F}_s]$ is deterministic. 
This implies for all $0\le s \le t \le T$ that  $E[ \eta^{-1}_s  \eta_t| \mathcal{F}_s]= E[\eta^{-1}_s  \eta_t]$ and completes the proof.
\end{proof}

\begin{remark}\label{rem:integrability_of_M}
    Assume that \ref{assumption:C1} is satisfied and that $M=(M_t)_{t \in [0,T]}$ given by $M_t=\eta_t (E[\eta_t])^{-1}$, $t\in[0,T]$, is a martingale. 
    Then applying Doob's maximal inequality yields that 
    $E[ \sup_{s\in[0,T]} \lVert M_s \rVert^2 ] 
        \le 4 E[\lVert M_T \rVert^2]
        \le 4 E[\lVert \eta_T \rVert^2] 
        \, \lVert (E[\eta_T])^{-1} \rVert^2 
        < \infty$.
\end{remark}

In the following result we provide an explicit solution $Y$ of the BSDE~\eqref{eq:BSDE_sing} 
and a deterministic candidate for the optimal control 
in the case where $\eta$ has uncorrelated multiplicative increments. 
The optimality of this candidate is subsequently established in \cref{prop:mult_uncorr_val_fun}.

\begin{propo}\label{prop:sol_bsde_uncor_mult}
    Assume that \ref{assumption:C0} and \ref{assumption:C1} are satisfied and that the process $\eta$ has uncorrelated multiplicative increments. 
    Let $Y\colon [0,T) \times \Omega \to \mathbb{S}_+^d$ be defined by 
    \begin{equation*}
        Y_t = \bigg(\int_t^T \big(E[\eta_s | \mathcal{F}_t]\big)^{-1} ds\bigg)^{-1}, 
        \quad t \in [0,T).
    \end{equation*}
    Moreover, denote 
    \begin{equation*}
        h_t = \int_t^T \big(E  [\eta_s]  \big)^{-1} ds, 
        \quad t\in[0,T].
    \end{equation*}
    Then:

    (i) It holds for all $t\in[0,T)$ that 
    $Y_t = \eta_t (E[\eta_t])^{-1} h_t^{-1}$.

    (ii) There exists a process $Z\colon [0,T) \times \Omega \to \mathbb{S}^d$ such that $(Y,Z)$ is a solution of the BSDE~\eqref{eq:BSDE_sing} in the sense of \cref{def:supersol_bsde}.

    (iii) Let $t \in [0,T)$ and $x \in \R^d$. 
    Define 
    \begin{align}\label{uncorr_mult_increm_X}
        & X_s=h_s h_t^{-1} x, \quad s \in [t,T],\\
        \label{uncorr_mult_increm_u} 
        & u_s=(E[\eta_s])^{-1} h_t^{-1} x, \quad s \in [t,T].
    \end{align} 
    Then 
    it holds that $u=(u_s)_{s\in[t,T]} \in \cA_0(t,x)$, and 
    $X=(X_s)_{s\in[t,T]}$ is the state~\eqref{eq:dynamics} associated to $u$.   
    Furthermore, 
    it holds that $J(t, x, u ) = \langle x, Y_t x\rangle$. 
\end{propo} 

\begin{proof}
Let $M_t=\eta_t (E[\eta_t])^{-1}$, $t\in[0,T]$, and note that from \eqref{eq:def_uncorr_mult_increm} and \cref{prop:process_mult_incr} we have for all $t\in[0,T)$ that 
\begin{align}\label{eq:repres_Y_uncorr_mult_increm}
    Y_t &= \bigg( \int_t^T \big( E[ \eta_t \eta^{-1}_t \eta_s | \cF_t] \big)^{-1} ds \bigg)^{-1} 
    =\eta_t \bigg( \int_t^T \big(E[  \eta^{-1}_t \eta_s | \cF_t]\big)^{-1} ds \bigg)^{-1}\\
    &= \eta_t \bigg( \int_t^T \Big( \big( E[\eta_t] \big)^{-1} E[\eta_s] \Big)^{-1} ds \bigg)^{-1}
    = \eta_t \big(E[\eta_t]\big)^{-1} \bigg( \int_t^T \big(E  [\eta_s]  \big)^{-1} ds \bigg)^{-1} 
    = M_t h_t^{-1}. \nonumber
\end{align}
This proves (i). 
Recall that $M=(M_t)_{t\in[0,T]}$ is a martingale by \cref{prop:process_mult_incr} and satisfies $E[\sup_{s\in[0,T]} \lVert M_s \rVert^2]<\infty$ due to \cref{rem:integrability_of_M}. 
The martingale representation theorem ensures that there exists a progressively measurable process $\widetilde{Z}\colon [0,T] \times \Omega \to \mathbb{R}^{d \times d}$ such that $E[\int_0^T \lVert \widetilde Z_s\rVert^2 ds]<\infty$ and $M_t=\int_0^t \widetilde Z_s dW_s$ for all $t\in[0,T]$.
For all $t \in [0,T)$ we define $Z_t=\widetilde Z_t h_t^{-1}$. Note that for all $t\in[0,T)$ we have $E[\int_0^t \lVert Z_s\rVert^2 ds]<\infty$. Moreover,~\eqref{eq:repres_Y_uncorr_mult_increm} and \cref{rem:integrability_of_M} guarantee for all $t\in[0,T)$ that $E[\sup_{s\in[0,t]} \lVert Y_s \rVert^2 ]<\infty$. 
Integration by parts and~\eqref{eq:repres_Y_uncorr_mult_increm} yield for all $t\in[0,T)$ that 
\begin{equation}\label{eq:Y_BSDE_uncorr_mult_incr}
    \begin{split}
        dY_t & = (dM_t) h_t^{-1} + M_t h_t^{-1} (E[\eta_t])^{-1} h_t^{-1} dt 
        = Z_t dW_t + Y_t \eta_t^{-1} Y_t dt . 
    \end{split}
\end{equation}
Since $Y$ is $\mathbb{S}_+^d$-valued, this shows that $Z=(Z_t)_{t\in[0,T)}$ is $\mathbb{S}^d$-valued. 
Furthermore, for any path of a 
process $(\psi_t)_{t \in [0,T]}$ that is left-continuous at $T$ with $\lim_{t \to T} \psi_t \neq 0$, it holds that 
$$\liminf_{t \to T}\langle \psi_t, Y_t \psi_t \rangle \ge \liminf_{t \to T} \big(\lambda_{\min}(Y_t) \|\psi_t \|^2\big) = \infty.$$ 
Thus the pair $(Y,Z)$ solves the BSDE~\eqref{eq:BSDE_sing} in the sense of \cref{def:supersol_bsde}, which proves~(ii).

To show (iii), 
we fix $t \in [0,T)$ and $x \in \R^d$. Due to \cref{lem:conv_of_inv} and \ref{assumption:C1}, it holds that $u\in L^1([t,T])$. 
Observe that $X_t=x$, $X_T=0$, and for all $s \in [t,T]$ it holds
$$\dot{X}_s = -(E[\eta_s])^{-1} h_t^{-1} x =-u_s.$$ 
Integration by parts, \eqref{eq:Y_BSDE_uncorr_mult_incr}, and \eqref{eq:repres_Y_uncorr_mult_increm} therefore imply for all $s \in [t,T)$ that  
\begin{equation*}
    \begin{split}
        d\langle X_s, Y_sX_s \rangle
        & = - \langle u_s, Y_s X_s \rangle ds 
        + \langle X_s, \widetilde Z_s h_s^{-1} X_s \rangle dW_s 
        + \langle X_s, M_s h_s^{-1} (E[\eta_s])^{-1} h_s^{-1} X_s \rangle ds \\
        & \quad - \langle X_s, Y_s u_s \rangle ds \\
        & = - \langle u_s, \eta_s (E[\eta_s])^{-1} h_s^{-1} h_s h_t^{-1} x \rangle ds
        + \langle X_s, Z_s X_s \rangle dW_s \\
        & \quad + \langle X_s, Y_s (E[\eta_s])^{-1} h_s^{-1} h_s h_t^{-1} x \rangle ds 
        - \langle X_s, Y_s u_s \rangle ds \\
        & = - \langle u_s, \eta_s u_s \rangle ds
        + \langle X_s, Z_s X_s \rangle dW_s . 
    \end{split}
\end{equation*}
Since $X$ is continuous and deterministic and for all $s \in [0,T)$ it holds $Z \in L^2(\Omega \times [0,s])$, we obtain for all $s\in [t,T)$ that 
$E[ \int_t^s \langle X_r,Z_r X_r \rangle dW_r | \cF_t]=0$. 
It follows for all $s \in [t,T)$ that 
\begin{equation}\label{eq:uncorr_mult_increm_J1}
    \begin{split}
        \langle x, Y_t x \rangle 
        & = E\bigg[ \int_t^s \langle u_r, \eta_r u_r \rangle dr \, \bigg|\, \cF_t \bigg] 
        + E\big[ \langle X_s, Y_s X_s \rangle \,\big| \, \cF_t \big] .
    \end{split}
\end{equation}
Note that \eqref{eq:repres_Y_uncorr_mult_increm} and the facts that $M$ is a martingale and $X$ is deterministic imply for all $s \in [t,T)$ that 
\begin{equation*}
    E\big[ \langle X_s, Y_s X_s \rangle \,\big| \, \cF_t \big]
    = E\big[ \langle X_s, M_s h_s^{-1} h_s h_t^{-1} x \rangle \,\big| \, \cF_t \big] 
    = \langle X_s, M_t h_t^{-1} x \rangle .
\end{equation*}
Since $X$ is continuous with $X_T=0$, we therefore obtain that 
\begin{equation}\label{eq:uncorr_mult_increm_J2}
    \lim_{s\to T} E\big[ \langle X_s, Y_s X_s \rangle \,\big| \, \cF_t \big]
    = \langle X_T, M_t h_t^{-1} x \rangle 
    = 0 .
\end{equation}
In addition, observe that the monotone convergence theorem shows that 
\begin{equation}\label{eq:uncorr_mult_increm_J3}
    \lim_{s \to T} E\bigg[ \int_t^s \langle u_r, \eta_r u_r \rangle dr \, \bigg|\, \cF_t \bigg] 
    = 
    E\bigg[ \int_t^T \langle u_r, \eta_r u_r \rangle dr \, \bigg|\, \cF_t \bigg] .
\end{equation}
We combine \eqref{eq:uncorr_mult_increm_J1}, \eqref{eq:uncorr_mult_increm_J2}, and \eqref{eq:uncorr_mult_increm_J3} to obtain 
$\langle x, Y_t x\rangle = J(t,x,u)$, which completes the proof. 
\end{proof}

\begin{propo}\label{prop:mult_uncorr_val_fun}
    Assume that 
    \ref{assumption:A0}, \ref{assumption:C0}, and \ref{assumption:C1} 
    are satisfied
    and that $\eta$ has uncorrelated multiplicative increments. 
    Let $(Y, Z)$ be the solution to the BSDE~\eqref{eq:BSDE_sing} constructed in \cref{prop:sol_bsde_uncor_mult}. 

    (i) 
    Let $x \in \R^d$, $t\in[0,T)$ 
    and let $u=(u_s)_{s\in[t,T)}$ be defined by 
    \eqref{uncorr_mult_increm_u}. 
    Then $u$ is the unique optimal control for~\eqref{eq:val_fct_sing} and  
    we have that $v(t, x) = \langle x, Y_t x\rangle$. 

    (ii)
    The solution $(Y,Z)$ is the minimal solution to the BSDE~\eqref{eq:BSDE_sing}.
\end{propo}

\begin{proof}
To prove (i), let $x \in \R^d$, $t\in[0,T)$, let $u=(u_s)_{s\in[t,T)}$ be defined by 
\eqref{uncorr_mult_increm_u}, and denote again $h_r=\int_r^T (E[\eta_s])^{-1} ds$, $r\in[0,T]$. 
Let $u'\in \cA_0(t,x)$ be an arbitrary control. Note that if $u' \notin L^2(\Omega \times [t, T])$, then the lower bound of $\eta$ implies that $u'$ produces infinite costs. Thus, it suffices to consider $u' \in L^2(\Omega \times [t, T])$. 
Let $X=(X_s)_{s\in[t,T]}$ and $X'=(X'_s)_{s\in[t,T]}$ be the state processes associated via~\eqref{eq:dynamics} to $u$ and $u'$, respectively. 
Note that for all $s \in [t,T]$ it holds  
$ \langle u_s, \eta_s u_s \rangle + \langle u'_s, \eta_s u'_s \rangle 
\ge 2 \langle \eta_s u_s, u'_s  \rangle.$ 
This and \eqref{uncorr_mult_increm_u} imply that 
\begin{equation}\label{eq:final_prop_uncorr_mult_increm1}
    \begin{split}
        & \int_t^T \big( \langle u'_s, \eta_s u'_s \rangle - \langle u_s, \eta_s u_s \rangle \big) \, ds
        \ge 
        2 \int_t^T \langle \eta_s u_s, u'_s - u_s \rangle\, ds \\
        & = 2 \int_t^T  (\eta_s u_s)^\top d(X'_s-X_s) 
        = 2 \int_t^T  \big(\eta_s (E[\eta_s])^{-1} h_t^{-1} x \big)^\top d(X'_s-X_s) .
    \end{split}
\end{equation}
Recall from \cref{prop:process_mult_incr} that $M=(M_s)_{s\in[0,T]}$ defined by $M_s=\eta_s (E[\eta_s])^{-1}$, $s\in[0,T]$, is a martingale. 
Since $X_t=x=X'_t$ and $X_T=0=X'_T$, we have by integration by parts that 
\begin{equation}\label{eq:final_prop_uncorr_mult_increm2}
    \begin{split}
        & \int_t^T  \big(M_s h_t^{-1} x \big)^\top d(X'_s-X_s)
        = -x^\top h_t^{-1} \int_t^T (dM_s^{\top})  (X'_s-X_s) .
    \end{split}
\end{equation}
Note that $E[\int_t^T \| u'_s \|^2 ds] < \infty$ implies that $E[\sup_{t \le s \le T} \| X'_s\|^2] < \infty$. 
This, the fact that~$X$ is continuous and deterministic, and \cref{rem:integrability_of_M} 
ensure that 
\begin{equation*}
    E\bigg[ \int_t^T (dM_s^\top) (X_s'-X_s) \,\bigg|\, \cF_t \bigg] = 0.
\end{equation*} 
Hence, \eqref{eq:final_prop_uncorr_mult_increm1} and \eqref{eq:final_prop_uncorr_mult_increm2} show that 
\begin{equation*}
    E\bigg[ \int_t^T \big( \langle u'_s, \eta_s u'_s \rangle - \langle u_s, \eta_s u_s \rangle \big) \, ds \,\bigg|\, \cF_t \bigg] 
    \ge 0. 
\end{equation*}
This together with \cref{prop:sol_bsde_uncor_mult} implies that $v(t,x)=\langle x, Y_t x\rangle =J(t,x,u)$. 

To prove (ii), suppose that $(Y',Z')$ is a solution of the BSDE~\eqref{eq:BSDE_sing} in the sense of \cref{def:supersol_bsde}. 
Then \cref{lemma:continuous_state} implies for all $x \in \R^d$ and $t \in [0,T)$ that $\langle x, Y'_t x \rangle \ge v(t,x) = \langle x, Y_t x\rangle$, which completes the proof. 
\end{proof}

\begin{ex}\label{ex:eta_for_uncorr_mult_increm}
Assume \ref{assumption:C0}.

(i) 
Let $\eta=(\eta_t)_{t \in [0,T]}$ be an $\mathbb{S}^d_{++}$-valued martingale that satisfies the integrability conditions \ref{assumption:C1}. 
Then $\eta$ has uncorrelated multiplicative increments and \cref{prop:sol_bsde_uncor_mult} gives a solution to the BSDE~\eqref{eq:BSDE_sing} in the sense of \cref{def:supersol_bsde}. 

(ii) 
To construct examples where also 
the condition that $\eta$ is bounded from below 
is satisfied, one can proceed by stopping:
Let $N=(N_t)_{t \in [0,T]}$ be a 
continuous, $\mathbb{S}^d$-valued 
martingale 
such that $E[\lVert N_T \rVert^2 ] < \infty$.
Let $\delta \in (0,\infty)$ and 
define the stopping time 
$\tau = \inf \{t \ge 0 \colon N_t \le \delta \Id \}$. 
Let $\eta=(\eta_t)_{t\in[0,T]}$ be given by 
$\eta_t := N_{t \wedge \tau}$, $t \in [0,T]$. 
Then it holds for all $t\in[0,T]$ that $\eta_t \ge \delta \Id$ and \ref{assumption:C1} is satisfied. 
Consequently, 
\cref{prop:mult_uncorr_val_fun} shows that 
the solution constructed in \cref{prop:sol_bsde_uncor_mult} is the minimal solution of the BSDE \eqref{eq:BSDE_sing} in the sense of \cref{def:supersol_bsde}.  
Moreover, since $\eta$ is a martingale, 
we have that 
the optimal strategy is linear, given by 
$X_s=\frac{T-s}{T-t}x$, $s\in [t,T]$.
\end{ex}

\subsection{The case when $\eta$ is an It\^o process} \label{ssect:eta_ito_process}

In this subsection, we assume that the process $\eta$ is an It\^o process and that all coefficients are bounded. More precisely:
\begin{enumerate}[label=\textbf{D\arabic*)}]
    \item \label{assumption:D0}
    There exist progressively measurable processes $b^{\eta}\colon [0,T] \times \Omega \to \mathbb{S}^d$ and $\sigma^{\eta}\colon [0,T] \times \Omega \to \mathbb{S}^d$ such that 
    $d\eta _t = b^\eta_t dt + \sigma^\eta_t dW_t$, $t\in[0,T]$, and there exists $K^{\eta} \in (0,\infty)$ such that $P$-a.s.\ for all $t\in[0,T]$ it holds $\lVert b^{\eta}_t \rVert \le K^{\eta}$ and $\lVert \sigma^{\eta}_t \rVert \le K^{\eta}$.
    
    \item \label{assumption:D1}
    It holds $\xi = \Id$ and 
    there exists $K \in (0, \infty)$ such that $P$-a.s.\ for all $t \in [0,T]$ it holds  $\| \lambda_t \| \le K$, $\| \phi_t \| \le K$, $\|A_t \| \le K$, 
    and 
    $\| \eta_t \| \le K$. 
\end{enumerate}

Note that under \ref{assumption:A0}, \ref{assumption:D0}, and \ref{assumption:D1} the conclusion of \cref{lem:lower_bound} holds.
The main result of this part is the following proposition, where we establish an asymptotic expansion near~$T$ of the minimal supersolution of the singular BSDE~\eqref{eq:BSDE_sing}.  

\begin{propo} \label{prop:asymp_expansion}
    Assume that \ref{assumption:A0}, \ref{assumption:D0}, and \ref{assumption:D1} are satisfied.
    Then there exists a progressively measurable process $H\colon [0,T] \times \Omega \to \mathbb{S}^d$ such that $P$-a.s.\ for all $t \in [0,T)$ the minimal supersolution $(Y,Z)$ to the BSDE~\eqref{eq:BSDE_sing} in the sense of \cref{def:supersol_bsde} admits the representation
    \begin{equation}\label{eq:asymp_expansion2}
        Y_t = \dfrac{1}{T-t} \eta_t + \dfrac{1}{(T-t)^2} H_t ,
    \end{equation}
    and there exists a constant $c \in (0, \infty)$ such that $P$-a.s.\ for all $t\in [0,T]$ we have 
    $$\|H_t\|\leq c(T-t)^2.$$
\end{propo}

As mentioned in the introduction of this section, this result generalizes the asymptotic expansion of~$Y$ to the multidimensional setting. Remark that it is obtained under the same assumptions as in the one-dimensional case (see \cite{caci:deni:popi:25} and \cite[Section 4.1]{grae:popi:21}). 

We split the proof of \cref{prop:asymp_expansion} into two lemmas. 
But first of all we reduce in \cref{rem:auxilary_rem_sec4.2} the problem to the case where $\phi$ and $A$ are the zero processes.
To this end, we introduce (under \ref{assumption:A0}, \ref{assumption:D0}, and \ref{assumption:D1}) the following auxiliary processes:

Let $\hat{A}=(\hat{A}_t)_{t\in[0,T]}$ be defined by 
$$\hat{A}_t= A_t + \eta_t^{-1} \phi_t, \quad t\in[0,T],$$ 
and note that \ref{assumption:A0} and \ref{assumption:D1} ensure that there exists $\delta_1 \in(0,\infty)$ such that $P$-a.s.\ for all $t\in[0,T]$ it holds $\lVert \hat{A}_t \rVert \le \delta_1$. 
Let $U=(U_t)_{t\in[0,T]}$ be the unique solution of 
$$dU_t=-U_t \hat{A}_t dt, \quad t\in[0,T], \quad U_0=\Id.$$ 
Note that since $\hat{A}$ is bounded, also $U$ and $U^{-1}$ are bounded. 
Furthermore, define $\widetilde{\lambda}=(\wt\lambda_t)_{t \in [0,T]}$ and $\widetilde{\eta}=(\wt\eta_t)_{t \in [0,T]}$ by 
\begin{equation*}
    \begin{split}
        &\wt\lambda_t= (U_t^\top)^{-1} (\lambda_t-\phi_t^\top \eta_t^{-1} \phi_t)U_t^{-1}, \quad t\in[0,T], \\
        & \wt\eta_t= (U_t^\top)^{-1} \eta_t U_t^{-1}, \quad t\in[0,T].
    \end{split}
\end{equation*}
Observe that \ref{assumption:A0} implies that $\wt\lambda$ is $\mathbb{S}_+^d$-valued, and \ref{assumption:A0} and \ref{assumption:D1} guarantee that there exist $\wt K, \wt \delta \in (0,\infty)$ such that $P$-a.s.\ for all $t\in[0,T]$ it holds $\lVert \wt\lambda_t \rVert \le \wt{K}$, $\lVert \wt\eta_t \rVert \le \wt{K}$, and $\wt \eta_t \ge \wt \delta \Id$. 
In addition, define $b^{\wt\eta}=(b^{\wt\eta}_t)_{t\in[0,T]}$ and $\sigma^{\wt\eta}=(\sigma^{\wt\eta}_t)_{t\in[0,T]}$ by 
\begin{equation*}
    \begin{split}
        &b^{\wt\eta}_t= (U_t^\top)^{-1} \big( b_t^{\eta} + \eta_t \hat{A}_t + \hat{A}_t^\top \eta_t \big) U_t^{-1}, \quad t\in[0,T], \\
        & \sigma^{\wt\eta}_t= (U_t^\top)^{-1} \sigma^{\eta}_t U_t^{-1}, \quad t\in[0,T].
    \end{split}
\end{equation*} 
By \ref{assumption:A0}, \ref{assumption:D0}, and \ref{assumption:D1} there exists $K^{\wt \eta} \in (0,\infty)$ such that $P$-a.s.\ for all $t\in[0,T]$ it holds $\lVert b^{\wt\eta}_t \rVert \le K^{\wt\eta}$ and $\lVert \sigma^{\wt\eta}_t \rVert \le K^{\wt\eta}$.
Moreover, it follows from \ref{assumption:D0} that $b^{\wt\eta}$ and $\sigma^{\wt\eta}$ are $\mathbb{S}^d$-valued and for all $t\in[0,T]$ it holds $d\wt \eta_t = b^{\wt \eta}_t dt + \sigma^{\wt\eta}_t dW_t$.

\begin{remark}\label{rem:auxilary_rem_sec4.2}
    Let \ref{assumption:A0}, \ref{assumption:D0}, and \ref{assumption:D1} be satisfied. 

    (i) 
    We remark that if $(Y,Z)$ is a solution to the BSDE~\eqref{eq:BSDE_sing} in the sense of \cref{def:supersol_bsde}, then $(\wt Y, \wt Z)$ given by $\wt Y = (U^\top)^{-1} Y U^{-1}$ and $\wt Z = (U^\top)^{-1} Z U^{-1}$ is a solution to the BSDE 
    \begin{equation}\label{eq:auxiliary_BSDE_sec4}
      d \widetilde Y_t  =  \widetilde Y_t  \widetilde \eta^{-1}_t   \widetilde Y_t dt -  \widetilde \lambda_t dt + \widetilde Z_t  d W_t, \quad t \in [0,T), \quad
        \wt Y_T=\infty \Ind_{C^c}.  
    \end{equation}
    in the sense of \cref{def:supersol_bsde}, and vice versa.

    (ii) 
    Let $(Y,Z)$ be the minimal supersolution to the BSDE~\eqref{eq:BSDE_sing} in the sense of \cref{def:supersol_bsde}.  
    Let $(\wt Y,\wt Z)$ be the minimal supersolution to the BSDE~\eqref{eq:auxiliary_BSDE_sec4} in the sense of \cref{def:supersol_bsde} and let $\wt H\colon [0,T] \times \Omega \to \mathbb{S}^d$ satisfy $\wt H_t = (T-t)^2 \wt Y_t - (T-t) \wt\eta_t$, $t\in[0,T)$. 
    Then we can show that $H\colon [0,T] \times \Omega \to \mathbb{S}^d$ defined by $H_t=U_t^\top \wt H_t U_t$, $t\in[0,T]$, satisfies~\eqref{eq:asymp_expansion2}.
\end{remark}

In view of \cref{rem:auxilary_rem_sec4.2} we consider in \cref{lem:existence_asymp_sol} and \cref{lem:min_asymp_sol} the case where $A$ and $\phi$ are zero. 
In the sequel, we consider the BSDE 
\begin{equation}  \label{eq:BSDE_H}
    \begin{split}
        & d\widetilde H_t = \bigg( \dfrac{1}{(T-t)^2} \widetilde H_t \widetilde \eta_t^{-1}\widetilde  H_t - (T-t)^2 \widetilde \lambda_t - (T-t)b^{\widetilde \eta}_t \bigg) dt + Z^{\widetilde H}_t dW_t, 
    \quad t \in [0,T], \\
    &\widetilde H_T = 0  \Id.
    \end{split}
\end{equation}

\begin{remark}\label{rem:HfromasympexpansionresultsatisfiesBSDE}
    In the setting of \cref{prop:asymp_expansion} with $A$ and $\phi$ equal to zero, the process $H$ satisfies for all $t \in [0,T)$ the dynamics in~\eqref{eq:BSDE_H} with $Z^H_t=-(T-t)\sigma^{\eta}_t+(T-t)^2 Z_t$, $t\in[0,T)$. 
\end{remark}

To prove \cref{prop:asymp_expansion}, we take the BSDE~\eqref{eq:BSDE_H} as a starting point.
We establish in \cref{lem:existence_asymp_sol} by a contraction argument that there exists a unique solution of~\eqref{eq:BSDE_H} in the following sense.

\begin{defi}\label{def:sol_bsde_sing_gen}
    We call a pair of progressively measurable processes $(H, Z^{H})\colon [0,T]\times \Omega \to \mathbb{S}^d \times \mathbb{S}^d$ a solution of the BSDE~\eqref{eq:BSDE_H} if 
    $$ \sup_{t \in [0,T)} \esssup_{\omega \in \Omega} \frac{\| H_t(\omega)  \|}{(T - t)^2} + E\bigg[\int_0^T \lVert Z^{H}_t \rVert^2 dt  \bigg] < \infty$$
    and if it holds $P$-a.s.\ for all $t \in [0,T]$ that 
    \begin{equation}\label{eq:bsde_sing_gen_eq}
        H_t = - \int_t^T \bigg( \frac{1}{(T-s)^2} H_s \wt \eta_s^{-1} H_s - (T-s)^2 \wt\lambda_s - (T-s) b_s^{\wt\eta} \bigg) ds - \int_t^T Z^H_s dW_s . 
    \end{equation}  
\end{defi}

Note in particular that this solution notion incorporates the desired property of $H$ in \cref{prop:asymp_expansion} that there exists a constant $c \in (0,\infty)$ such that $P$-a.s.\  for all $t\in[0,T]$ it holds $\lVert H_t \rVert \le c (T-t)^2$. 
Moreover, we remark that the singularity has shifted from the terminal condition in~\eqref{eq:BSDE_sing} to the driver in~\eqref{eq:BSDE_H}. 

Once we have established uniqueness and existence of~\eqref{eq:BSDE_H}, we use the solution component $H$ of the BSDE~\eqref{eq:BSDE_H} to define $Y$ via~\eqref{eq:asymp_expansion2}.
We show in \cref{lem:min_asymp_sol} that $Y$ defined like this is equal to the first component of the minimal solution to the singular BSDE~\eqref{eq:BSDE_sing} with $A$ and $\phi$ equal to zero.

\begin{lemma}\label{lem:existence_asymp_sol}
Assume that \ref{assumption:A0}, \ref{assumption:D0}, and \ref{assumption:D1} are satisfied and that $A$ and $\phi$ are zero.  
Then there exists a unique solution $(H, Z^{H})$ to the BSDE~\eqref{eq:BSDE_H} in the sense of \cref{def:sol_bsde_sing_gen}.
\end{lemma}

\begin{proof}
Let us define the process $\gamma=(\gamma_t)_{t \in [0,T]}$ by 
$\gamma_t = (T-t)\lambda_t +  b^{\eta}_t$, $t\in [0,T]$.
Note that \ref{assumption:D0} and \ref{assumption:D1} ensure that the process $\gamma$ is $\mathbb S^d$-valued and that there exists $K^{\gamma} \in (0,\infty)$ such that $P$-a.s.\ for all $t \in [0,T]$ we have $\lVert \gamma_t \rVert \le K^{\gamma}$.  
For all $\tau \in (0,T)$ we define $\mathcal H_\tau$ as the set of all adapted and bounded processes 
$H\colon [T-\tau,T] \times \Omega \to \mathbb{S}^d$ such that there exists $c > 0$ such that for all $t\in [T-\tau,T]$ it holds $\|H_t\|\leq c (T-t)^2$. We endow $\mathcal H_\tau$ with the norm 
$$\| H\|_{\mathcal H_\tau} = \sup_{t \in [T-\tau,T)} \esssup_{\omega \in \Omega} \dfrac{\|H_t(\omega)\|}{(T-t)^2}.$$
Furthermore, for all $\tau \in (0,T)$ and $R \in (0,\infty)$ we define $\mathcal{H}_{\tau,R}$ as the set of all processes $H\in\mathcal{H}_{\tau}$ that satisfy $\| H\|_{\mathcal H_\tau} \le R$. 
On $\mathcal H_\tau$, we define the operator $\Gamma$ as follows: for all $H\in\mathcal{H}_{\tau}$ and $t\in [T-\tau,T]$ let 
$$\Gamma(H)_t= E \bigg[\int_t^T \bigg((T-s)\gamma_s - \dfrac{1}{(T-s)^2} H_s \eta_s^{-1}H_s  \bigg) ds \,\bigg|\, \mathcal F_t \bigg].$$
Note that for all $t\in [T-\tau,T]$ it holds
$$\|\Gamma(0)_t\|  \leq E\bigg[\int_t^T (T-s) \|\gamma_s \| ds \,\bigg|\, \mathcal F_t \bigg] \leq \dfrac{K^{\gamma}}{2} (T-t)^2.$$
Thus, 
\begin{equation}\label{eq:boundGammaZero}
    \|\Gamma(0)\|_{\mathcal H_\tau} \leq \dfrac{K^{\gamma}}{2}. 
\end{equation}
For all $H, \wh H \in \cH_{\tau}$ set $\Delta H = H - \widehat H$ 
and note that for all $t\in [T-\tau, T]$ it holds 
\begin{align*}
& \Gamma(H)_t - \Gamma(\widehat H)_t \\ 
& = -  E\bigg[\int_t^T \dfrac{1}{(T-s)^2} \bigg( H_s \eta_s^{-1}H_s - \widehat H_s \eta_s^{-1} \widehat H_s \bigg) ds \,\bigg|\, \mathcal F_t \bigg] \\
& = - E\bigg[\int_t^T \dfrac{1}{(T-s)^2}  \bigg((\Delta H_s) \eta_s^{-1} \widehat H_s + \widehat H_s \eta_s^{-1} (\Delta H_s) 
+\Delta H_s \eta_s^{-1} \Delta H_s 
\bigg) ds \,\bigg|\, \mathcal F_t \bigg] .
\end{align*}
Since $P$-a.s.\ for all $t\in[0,T]$ it holds $\lVert \eta_t^{-1} \rVert \le \sqrt{d} \delta^{-1}$, 
we obtain for all $H, \wh H \in \cH_{\tau,R}$ that $P$-a.s.\ for all $t\in [T-\tau, T]$ it holds  
\begin{align*}
\| \Gamma(H)_t - \Gamma(\widehat H)_t \| & \leq (T-t)^3 
\dfrac{4 R \sqrt{d}}{3\delta}  
\|\Delta H\|_{\mathcal H_\tau}.
\end{align*}
It follows for all $H, \wh H \in \cH_{\tau,R}$ that 
\begin{equation}\label{eq:estimateDiffGammaHwhH}
    \| \Gamma(H) - \Gamma(\widehat H) \|_{\mathcal H_\tau} 
    \leq \tau
    \dfrac{4R\sqrt{d}}{3\delta} \|\Delta H\|_{\mathcal H_\tau} .
\end{equation}
This and \eqref{eq:boundGammaZero} imply for all $H\in\cH_{\tau,R}$ that 
\begin{equation}\label{eq:estimateGammaHtaunorm}
    \| \Gamma(H)  \|_{\mathcal H_\tau}
    \leq \| \Gamma(H) - \Gamma(0) \|_{\mathcal H_\tau} + \| \Gamma(0)  \|_{\mathcal H_\tau} 
    \leq \tau 
    \dfrac{4R \sqrt{d}}{3\delta}
    R + \dfrac{K^{\gamma}}{2} .
\end{equation}
Choose $R = K^{\gamma}$ and $\tau \in (0,T)$ such that 
$\tau  
\frac{4K^{\gamma} \sqrt{d}}{3\delta} 
\leq \frac{1}{2}$. 
We then have from~\eqref{eq:estimateGammaHtaunorm} and~\eqref{eq:estimateDiffGammaHwhH} 
for all $H, \wh H \in \cH_{\tau,R}$ that 
$\| \Gamma(H)  \|_{\mathcal H_\tau}\leq R$ and 
$\| \Gamma(H) - \Gamma(\widehat H) \|_{\mathcal H_\tau}  \leq \frac{1}{2} \|\Delta H\|_{\mathcal H_\tau}$. 
By a fixed point argument, there exists a unique process $H' \in \cH_{\tau,R}$ such that $H'=\Gamma( H')$.  
Moreover, the martingale representation theorem ensures that there exists a progressively measurable process $Z^{H'}\colon [T-\tau,T] \times \Omega \to \mathbb{S}^d$ such that $E[\int_{T-\tau}^T \lVert Z^{H'}_s \rVert^2 ds ] < \infty$ and $P$-a.s.\ for all $t\in[T-\tau,T]$ it holds 
\begin{equation*}
    E \bigg[\int_{T-\tau}^T \bigg((T-s)\gamma_s - \dfrac{1}{(T-s)^2} H'_s \eta_s^{-1} H'_s  \bigg) ds \,\bigg|\, \mathcal F_t \bigg]
    = 
    H'_{T-\tau} + \int_{T-\tau}^t Z_s^{H'} dW_s .
\end{equation*}
It follows that $(H',Z^{H'})$ is a solution of the BSDE~\eqref{eq:BSDE_H} on $[T-\tau,T]$. 

We next establish existence on the whole interval $[0,T]$. 
Since $H' \in \cH_{\tau,R}$, we have $P$-a.s.\ for all $t\in[T-\tau,T)$ that 
$H'_{t} \ge - R (T-t)^2 \Id$. 
This and \ref{assumption:A0} guarantee that there exists $t_0 \in [T-\tau,T)$ such that 
$$\dfrac{1}{T-t_0} \eta_{T-t_0} + \dfrac{1}{(T-t_0)^2} H'_{T-t_0} 
\ge \Big(\frac{\delta}{T-t_0} - R \Big) \Id 
\ge 0.$$
Consider the BSDE~\eqref{eq:BSDE_pen2} with $T$ replaced by $T-t_0$ and with $A=0$, $\phi=0$, $\bar{\xi}= \frac{1}{T-t_0} \eta_{T-t_0} + \frac{1}{(T-t_0)^2} H'_{T-t_0}$. 
By \cite[Theorem~6.1]{sunxiongyong2021indefinite} (see also \cref{prop:sol_pen_problem}) there exists a unique solution $(Y,Z)$ of this BSDE in the sense of \cref{def:sol_bsdes} and $Y$ is bounded. 
Define $H \colon [0,T] \times \Omega \to \mathbb{S}^d$ by $H_t=H'_t$, $t\in[T-t_0,T]$, and 
$H_t = (T-t)^2 Y_t - (T-t) \eta_t$, $t\in[0,T-t_0)$. 
Moreover, define $Z^H \colon [0,T] \times \Omega \to \mathbb{S}^d$ by $Z^H_t=Z^{H'}_t$, $t\in[T-t_0,T]$, and 
$Z^H_t = -(T-t)\sigma^{\eta}_t+(T-t)^2 Z_t$, $t\in[0,T-t_0)$. 
It then holds that $(H,Z^H)$ is a solution of the BSDE~\eqref{eq:BSDE_H} in the sense of \cref{def:sol_bsde_sing_gen}. 

To finish the proof, let us show that uniqueness holds. 
Suppose that $(\wh H, Z^{\wh H})$ is another solution of the BSDE~\eqref{eq:BSDE_H} in the sense of \cref{def:sol_bsde_sing_gen}, and recall the notation 
$\Delta H = H -  \widehat H$.  
Note that for all $t \in [0,T]$ it holds 
    \begin{align*}
        d(\Delta H_t) 
        & = \! \bigg( \dfrac{1}{(T-t)^2} (\Delta H_t) \eta_t^{-1} (\Delta H_t)  
        + 
        \dfrac{1}{(T-t)^2} (\Delta H_t ) \eta_t^{-1} \wh H_t 
        + 
        \dfrac{1}{(T-t)^2} \wh H_t \eta_t^{-1} 
        (\Delta H_t) 
        \! \bigg) dt  \\
        & \quad + (Z^{H}_t - Z^{\wh H}_t) dW_t .
    \end{align*}
Let $\Theta = (\Theta_t)_{t\in[0,T]}$
and $\psi=(\psi_t)_{t \in [0,T]}$ be defined by 
$\Theta_T=0$, $\psi_T=0$, and 
   $$
        \Theta_t = \dfrac{1}{(T-t)^2} \Delta H_t \eta_t^{-1} \Delta H_t, \qquad 
        \psi_t  = \dfrac{1}{(T-t)^2} \eta_t^{-1} \wh H_t ,\quad t \in [0,T).  
   $$ 
Observe that $\Theta$ is $\mathbb{S}_+^d$-valued. 
Moreover, note that $\psi$ is bounded since there exists $\wh c \in (0,\infty)$ such that $P$-a.s.\ for all $t\in[0,T]$ it holds $\lVert \wh H_t \rVert \le \wh c (T-t)^2$. Similarly, also $\Theta$ is bounded. 
Furthermore, for all $t \in [0,T]$ let $(V_{t,s})_{s \in [t,T]}$ be the unique solution of $V_{t,s}=\Id - \int_t^s \psi_r V_{t,r} dr$, $s \in [t,T]$, and note that $E[\sup_{s \in [t,T]} \lVert V_{t,s} \rVert^4 ]<\infty$.
Integration by parts shows for all $t \in [0,T]$ that 
    \begin{equation}\label{eq:DeltaHwithV}
        \Delta H_t 
        = - \int_t^T V_{t,r}^\top \Theta_r V_{t,r} dr 
        - \int_t^T V_{t,r}^\top (Z^H_r - Z^{\wh H}_r) V_{t,r} dW_r .
    \end{equation}
Since $E[\sup_{s \in [t,T]} \lVert V_{t,s} \rVert^4 ]<\infty$ for all $t \in [0,T]$ and $Z^H, Z^{\wh H} \in L^2(\Omega \times [0,T])$, we have for all $t\in[0,T]$ that 
$E[ \int_t^T V_{t,r}^\top (Z^H_r - Z^{\wh H}_r) V_{t,r} dW_r \,|\,\cF_t]=0$.
This, \eqref{eq:DeltaHwithV}, and the fact that $\Theta$ is $\mathbb{S}_+^d$-valued imply for all $t \in [0,T]$ that $\Delta H_t \le 0$.    
Vice versa, we obtain $H = \widehat H$.
This moreover implies that $Z^H=Z^{\wh H}$. 
\end{proof}

\begin{lemma}\label{lem:min_asymp_sol}
   Assume that \ref{assumption:A0}, \ref{assumption:D0}, and \ref{assumption:D1} are satisfied and that $A$ and $\phi$ are zero. 
    Let $(H,Z^H)$ be the unique solution of the BSDE~\eqref{eq:BSDE_H} in the sense of \cref{def:sol_bsde_sing_gen}. 
    Let $Y=(Y_t)_{t\in[0,T)}$ be defined by~\eqref{eq:asymp_expansion2} and let $(Y',Z')$ be the minimal solution of the singular BSDE~\eqref{eq:BSDE_sing} in the sense of \cref{def:supersol_bsde}. Then it holds that $Y=Y'$. 
\end{lemma}

\begin{proof}
    Let $Z=(Z_t)_{t \in [0,T)}$ be defined by 
    $$Z_t = \dfrac{1}{T-t} \sigma_t^{\eta} + \dfrac{1}{(T-t)^2} Z_t^H, \quad  t \in [0,T).$$
    Then we can show that the pair $(Y,Z)$ satisfies~\eqref{eq:supersoln_BSDE_equation}. 
    Moreover, it holds for all $t \in [0,t]$ that $Z \in L^2(\Omega \times [0,t])$. 
    Note that there exists $c \in (0,\infty)$ such that $P$-a.s.\ for all $t\in [0,T]$ it holds 
    \begin{equation}\label{eq:property_H}
        \| H_t\| \leq c(T-t)^2 .
    \end{equation}
    This and \ref{assumption:D1} ensure for all $t \in [0,T)$ that  $(Y_s)_{s \in [0,t]}$ is bounded.  
    In addition, we obtain from \eqref{eq:property_H} and \ref{assumption:A0} for all $x \in \mathbb R^d\setminus\{0\}$ and $t \in (\max\{0,T-\frac{\delta}{c}\},T)$ that
    $$\langle x, Y_t x \rangle \geq \Big( \dfrac{\delta}{T-t} - c\Big)  \| x \|^2 
    > 0.$$
    Thus there exists $\tau \in (0,T)$ such that $Y_t$ is $\mathbb S^d_{++}$-valued for all $t\in [T-\tau,T)$.
    Since $Y_{T-\tau}$ is bounded and $\mathbb S^d_{++}$-valued, and due to \ref{assumption:A0} and \ref{assumption:D1}, we can apply \cite[Theorem~6.1]{sunxiongyong2021indefinite} to deduce via non-negativity of the associated cost functional that $Y_t$ is $\mathbb S^d_{+}$-valued for all $t\in [0,T-\tau]$.  
    Observe that for any path of a  
    process $(\psi_t)_{t \in [0,T]}$ that is left-continuous at $T$ with $\lim_{t \to T} \psi_t \neq 0$, it holds that 
    $$\liminf_{t \to T}\langle \psi_t, Y_t \psi_t \rangle \ge \liminf_{t \to T} \big(\lambda_{\min}(Y_t) \|\psi_t \|^2\big) = \infty.$$ 
    We conclude that $(Y,Z)$ is a solution of the singular BSDE~\eqref{eq:BSDE_sing} in the sense of \cref{def:supersol_bsde}. 
    It follows $P$-a.s.\ for all $t \in [0,T)$ that $Y_t \ge Y'_t$.

    Next, we define the process $H'=(H'_t)_{t \in [0,T]}$ by $H'_T=0$ and  
    $$H'_t = (T-t)^2 Y'_t - (T-t) \eta_t, \quad t\in [0,T).$$
    Moreover, let $Z^{H'}=(Z^{H'}_t)_{t \in [0,T]}$ be defined by $Z^{H'}_T=0$ and 
    \begin{equation}\label{eq:defZHprime}
        Z^{H'}_t=-(T-t)\sigma^{\eta}_t+(T-t)^2 Z'_t, \quad t\in[0,T).
    \end{equation}
    Observe that if $H=H'$, then $Y=Y'$. 
    We therefore show in the remainder of this proof that $(H',Z^{H'})$ is a solution of the BSDE~\eqref{eq:BSDE_H} in the sense of \cref{def:sol_bsde_sing_gen}.  
    Note that by \ref{assumption:D1} we have $P$-a.s.\ for all $t \in [0,T]$ that $H'_t \ge - K (T-t) \Id$.
    This, the fact that $Y\ge Y'$, and \eqref{eq:property_H} imply $P$-a.s.\ for all $t\in[0,T]$ that 
    \begin{equation}\label{eq:asymp_aux1}
         - K (T-t) \Id \leq H'_t \leq H_t \leq c (T-t)^2 \Id.
    \end{equation}     
    From \cref{lem:lower_bound} and \cref{lem:conv_of_inv}
    we know for all $t \in [0,T)$ that
    \begin{equation}\label{eq:lowerboundYasympexpansion}
        Y'_t \geq \bigg(  E\bigg[ \int_t^T \eta_s^{-1} ds \,\bigg|\, \cF_t\bigg]\bigg)^{-1}. 
    \end{equation}
    Integration by parts shows for all $0\le t\le s \le T$ that 
    \begin{equation*}
        \eta_s^{-1} = \eta_t^{-1} +\int_t^s \eta_r^{-1}\big(  \sigma_r^\eta \eta_r^{-1} \sigma_r^\eta -  b_r^\eta \big)\eta_r^{-1} dr - \int_t^s \eta_r^{-1} \sigma_r^\eta \eta_r^{-1} dW_r.
    \end{equation*}
    Since $\eta^{-1}$ and $\sigma^{\eta}$ are bounded, this and Fubini's theorem imply for all $t \in [0,T]$ that   
    \begin{equation}\label{eq:cond_exp_of_eta}
        \begin{split}
        E\bigg[ \int_t^T \eta_s^{-1} ds \, \bigg| \, \cF_t\bigg] 
        & = (T-t) \eta_t^{-1} + E\bigg[ \int_t^T \int_t^s \eta_r^{-1}\big(  \sigma_r^\eta \eta_r^{-1} \sigma_r^\eta -  b_r^\eta \big)\eta_r^{-1} \,dr\, ds \,\bigg|\, \cF_t\bigg] \\
        & = (T-t) \eta_t^{-1} +E\bigg[ \int_t^T (T-r) \eta_r^{-1}\big(  \sigma_r^\eta \eta_r^{-1} \sigma_r^\eta -  b_r^\eta   \big)\eta_r^{-1} \, dr  \,\bigg|\, \cF_t\bigg] .
        \end{split}
    \end{equation}
    For all $t \in [0,T)$ define 
    $$B_t = \dfrac{1}{T-t} E\bigg[ \int_t^T (T-r) \eta_r^{-1}\big(  \sigma_r^\eta \eta_r^{-1} \sigma_r^\eta -  b_r^\eta   \big)\eta_r^{-1} \, dr  \,\bigg|\, \cF_t\bigg] \eta_t .$$
    It follows from \eqref{eq:lowerboundYasympexpansion} and \eqref{eq:cond_exp_of_eta} for all $t \in [0,T)$ that 
    \begin{equation*}
        Y'_t \ge ( (T-t) \eta_t^{-1} + (T-t) B_t  \eta_t^{-1} )^{-1} .
    \end{equation*}
    By \ref{assumption:A0}, \ref{assumption:D0}, and \ref{assumption:D1} there exists $b \in (0, \infty)$ such that $P$-a.s.\ for all $t\in[0,T)$ it holds $\|B_t\| \leq \frac{b}{2} (T-t)$. 
    Thus there exists $\tau' \in (0,T)$ such that $P$-a.s.\ for any $t \in [T-\tau',T)$ we have that $\Id + B_t$ is invertible with  
    $$\left( \Id +  B_t\right)^{-1} - \Id = B_t \sum_{k=1}^{\infty} (-1)^k B_t^{k-1}. $$
    Hence, 
    there exists $b' \in (0,\infty)$ such that $P$-a.s.\ for all $t \in [T-\tau',T)$ it holds  
    $$\dfrac{1}{(T-t)^2}H'_t 
    \ge \dfrac{1}{T-t} \eta_t [( \Id +  B_t)^{-1} - \Id ] \ge - b' \Id.$$ 
    Combining this with \eqref{eq:asymp_aux1} 
    shows that there exists $c' \in (0,\infty)$ such that $P$-a.s.\ for all $t \in [0,T]$ it holds $\lVert H'_t \rVert^2 \le c'(T-t)^2$.  
    Moreover, observe that both $H'$ and $Z^{H'}$ are $\mathbb{S}^d$-valued. 
    It\^o's formula shows for all $r \in (0,T)$ that 
    \begin{align} \nonumber
            (T-r)^4 \lVert Y'_{r} \rVert^2 
            & = \lVert Y'_0 \rVert^2  
             - 4 \int_0^{r} (T-s)^3 \lVert Y'_s\rVert^2 ds 
            + \int_0^{r} (T-s)^4 \lVert Z'_s \rVert^2 ds \\ \nonumber
            & \quad + 2 \int_0^{r} (T-s)^4 \tr( Y'_s Z'_s) dW_s 
            - 2 \int_0^{r} (T-s)^4 \tr(Y'_s \lambda_s ) ds \\
            & \quad + 2 \int_0^{r} (T-s)^4 \tr( Y'_s Y'_s \eta_s^{-1} Y'_s ) ds .
            \label{eq:itointegrabilityz}
    \end{align}
    Moreover, 
    note that 
    \eqref{eq:asymp_aux1} and \ref{assumption:D1} imply $P$-a.s.\ for all $t \in [0,T)$ that 
    $$0 \le Y_t'\le \frac{1}{T-t} K \Id + c \Id.$$ 
    This, \ref{assumption:A0}, \ref{assumption:D1}, the fact that $Z' \in L^2(\Omega \times [0,t])$ for any $t \in [0,T)$, and \eqref{eq:itointegrabilityz} ensure that there exists $a' \in (0,\infty)$ such that for all $r \in (0,T)$ it holds 
    \begin{equation*}
        E\bigg[ \int_0^{r} (T-s)^4 \lVert Z'_s \rVert^2 ds \bigg]
        \le a' .
    \end{equation*}
    Fatou's lemma implies that $E[ \int_0^{T} (T-s)^4 \lVert Z'_s \rVert^2 ds ]\le a'$.
    This, \ref{assumption:D0}, and \eqref{eq:defZHprime} prove that $Z^{H'} \in L^2(\Omega \times [0,T])$. 
    Furthermore, we can show that the pair $(H',Z^{H'})$ satisfies~\eqref{eq:bsde_sing_gen_eq}. 
    Now, \cref{lem:existence_asymp_sol} completes the proof.
\end{proof}

\begin{proof}[Proof of \cref{prop:asymp_expansion}] 
    Due to \cref{rem:auxilary_rem_sec4.2}, it suffices to consider the case where $A$ and $\phi$ are zero. The claim now follows from \cref{lem:min_asymp_sol}, where the existence of a unique solution to the BSDE~\eqref{eq:BSDE_H} in the sense of \cref{def:sol_bsde_sing_gen} is ensured by \cref{lem:existence_asymp_sol}.
\end{proof}

\begin{remark}\label{rem:asymptot_expansion_uncorr_mult_increm}
Assume that \ref{assumption:A0}, \ref{assumption:D0}, and \ref{assumption:D1} are satisfied. 
In the framework of \cref{ssect:umi} (i.e., we additionally assume \ref{assumption:C0} and \ref{assumption:C1}), the BSDE~\eqref{eq:BSDE_H} becomes 
\begin{equation}\label{eq:BSDEHrem}
    dH_t = \bigg( \dfrac{1}{(T-t)^2} H_t \eta_t^{-1}H_t  - (T-t)b^{\eta}_t  \bigg) dt + Z^H_t dW_t, \quad t \in [0,T], \quad H_T=0.
\end{equation}

(i) 
If $\eta$ is an $\mathbb{S}_{++}^d$-valued continuous  martingale ($b^\eta=0$), then the solution of~\eqref{eq:BSDEHrem} is given by $(H,Z^H)=(0,0)$. Thus we recover that $Y_t = \eta_t / (T-t)$, $t\in[0,T)$ (see \cref{prop:sol_bsde_uncor_mult}). 

(ii) 
More generally, suppose that $\eta$ has uncorrelated multiplicative increments. 
Then for all $t \in [0,T]$ the drift $b^\eta_t$ is equal to $\eta_t g(t)$, where $g$ is a deterministic matrix-valued function (see \cite[Lemma~11]{caci:deni:popi:25} with straightforward modifications).  
Let $H_T=0$ and $H_t = (T-t) \eta_t h(t)$, $t\in[0,T)$, where 
$$h(t) = -\Id + (T-t) \left(  \int_t^T G(s)^{-1}G(t)  ds \right)^{-1}, \quad t\in[0,T),$$
and $dG(t)=G(t) g(t) dt$, $t\in[0,T]$, $G(0)=\Id$. 
Then $H=(H_t)_{t \in [0,T]}$ is the first component of the solution to~\eqref{eq:BSDEHrem}, 
see \cite[Lemma~13]{caci:deni:popi:25} with straightforward modifications. It follows from~\eqref{eq:asymp_expansion2} for all $t\in [0,T)$ that 
$$Y_t= \dfrac{\eta_t}{T-t} + \dfrac{H_t}{(T-t)^2} = \eta_t  \bigg(  \int_t^T G(s)^{-1}G(t) ds \bigg)^{-1}.$$
Since $E[\eta_s | \mathcal F_t]=\eta_t G(t)^{-1}G(s)$, $0 \le t \le s \le T$, we again obtain the result of \cref{prop:sol_bsde_uncor_mult}.
\end{remark}

\begin{remark}
Assume that \ref{assumption:A0}, \ref{assumption:D0}, and \ref{assumption:D1} are satisfied. 
Fix $t \in [0, T)$ and $x \in \R^d$. 
By \cref{prop:sol_sing_problem} and \cref{prop:asymp_expansion} the optimal state process for~\eqref{eq:val_fct_sing} satisfies 
\begin{equation*}
\begin{split}
    \dot X_s = - \eta^{-1}_s \left( \frac{\eta_s}{(T-s)} + \frac{H_s}{(T-s)^2} - \phi_s  \right) X_s + A_s X_s, \quad s \in [t, T), \quad X_t = x.
\end{split}
\end{equation*}
Thus, we have for all $s \in [t,T)$ that 
\begin{equation*}
\begin{split}
    d \|X_s \|^2 = \Big\langle X_s,\Big(&-\frac{2}{(T-s)} - \frac{1}{(T-s)^2}(\eta^{-1}_s H_s 
    + (\eta^{-1}_s H_s)^\top)\\
    &+ (A^\top + A + \eta^{-1}_s \phi_s + \phi_s^\top \eta^{-1}_s) \Big)X_s\Big\rangle  ds .
\end{split}
\end{equation*}
By \cref{prop:asymp_expansion} 
there exists a constant $c_1 \in (0,\infty)$ such that for all $s \in [0,T]$ it holds $\| H_s  \| \le c_1 (T-s)^2 $.
Moreover, remark that for all $s \in [0,T]$ it holds $\eta_s \ge \delta \Id$, $\| A_s \| \le K$ and $\| \phi_s \| \le K $.
Let $c_2 = 2 (\delta^{-1} \sqrt{d} c_1 + K + \delta^{-1} \sqrt{d} K)$. 
Then we have for all $s \in [t,T)$ that
\begin{equation*}
    \lVert X_s \rVert^2 \le \lVert x \rVert^2 
    + \int_t^s \Big( - \frac{2}{T-r} + c_2 \Big) \lVert X_r \rVert^2 dr .
\end{equation*}
Gronwall's lemma yields for all $s \in [t, T)$ that 
\begin{equation*}
    \lVert X_s \rVert^2 \le \lVert x \rVert^2 
    \exp\bigg( \int_t^s \Big( - \frac{2}{T-r} + c_2 \Big) dr \bigg) 
    \le \| x \|^2\frac{(T-s)^2}{(T-t)^2} e^{C_2T} .
\end{equation*}
It follows that $\| X_s \| = O(T-s)$ as $s \to T$.  
Moreover, we derive that $\lim_{s \to T} \langle X_s, Y_s X_s \rangle = 0$ and, by dominated convergence, $\lim_{s \to T} E[\langle X_s, Y_s X_s \rangle ] = 0$. Hence, when all coefficients are bounded, the asymptotic behavior of the optimal state is similar to that in the one-dimensional case.
\end{remark}

\appendix
\section{Two elementary lemmas}

In this appendix we provide two elementary lemmas that we use in the proofs. 
The first one helps in establishing matrix convergence. 
 
\begin{lemma}\label{lemma:matrix_convergence}
Let $\xi^n$, $n \in \N$, be a monotone and bounded sequence of matrices in $\mathbb S^d$. Then there exists the limit $\lim_{n \to \infty} \xi^n = \xi$.  
\end{lemma}

\begin{proof}
    To establish the claim we use the polarization property that for all $i,j \in \{1,\ldots, d \}$  and $n\in\N$ it holds
    $$\langle e_i,\xi^n e_j\rangle = \tfrac{1}{4}(\langle e_i+e_j,\xi^n(e_i+e_j)\rangle - \langle e_i-e_j,\xi^n(e_i-e_j)\rangle ).$$ 
    The existence of the limits of both terms on the right-hand side follows from monotonicity and boundedness.
\end{proof}

The second lemma shows the elementary inequality $E[X^{-1}] \ge (E[X])^{-1}$ for random variables $X$ taking values in the set of symmetric positive definite matrices. We give an elementary, self-contained proof that avoids using convexity of the inversion operator or an operator Jensen inequality.

\begin{lemma}\label{lem:conv_of_inv}
    Let $X$ be a random variable taking values in $\mathbb{S}^d_{++}$ with $E[\| X \|] < \infty$ and $E[\| X^{-1} \|] < \infty$. Then we have $E[X^{-1}] \ge (E[X])^{-1}$.   
\end{lemma}

\begin{proof}
    Let $v \in \R^d$ and set $u = (E[X])^{-1}v$. Then we have
    \begin{equation*}
    \begin{split}
        &\left\langle v, \left(E[X^{-1}] - (E[X])^{-1} \right) v \right\rangle \\
        &= \langle v, E[X^{-1}] v \rangle - 2 \langle v, (E[X])^{-1}v \rangle + \langle (E[X])^{-1}v, E[X] (E[X])^{-1}v \rangle \\
        &= E\left[ \langle v, X^{-1} v \rangle - 2 \langle u,v\rangle + \langle u, X u \rangle\right]= E\Big[\big\| X^{-\frac{1}{2}} v - X^{\frac{1}{2}} u \big\|^2\Big]
        \ge 0.
    \end{split}
    \end{equation*}
 This completes the proof.   
\end{proof}

\section*{Acknowledgments}

All authors acknowledge funding by the Deutsche Akademische Austauschdienst (DAAD, German Academic Exchange Service) within the funding program ``Programme for Pro\-ject-Related Personnel Exchange with France'' -- Project number 57702378, and by Campus France PHC (Projet Hubert Curien) Procope number 50835ZC. 
Julia Ackermann and Thomas Kruse acknowledge funding by the Deutsche For\-schungsgemeinschaft (DFG, German Research Foundation) -- Project-ID 531152215 -- CRC 1701.

\bibliographystyle{abbrv}
\bibliography{literature}

\end{document}